\documentclass{amsart}

\usepackage[T1]{fontenc}
\usepackage[utf8]{inputenc}
\usepackage[USenglish]{babel}
\usepackage{textcase}

\usepackage[
	bookmarks=true,
	plainpages=false,
	linktocpage,
	colorlinks=true,
	citecolor=green!80!black,
	linkcolor=red!70!black,
	filecolor=magenta,
	urlcolor=magenta,
	breaklinks,
	pdfauthor={Martin Winter},
]{hyperref}

\usepackage{amsmath,amsthm,amssymb}
\usepackage{calc, mathtools} %calc for \widthof, mathtools for \mathclap

\usepackage{colortbl,color,xcolor} % colors
\usepackage{centernot} % crossing out symbols correctly
\usepackage{array} % arrays
\usepackage{enumitem,moreenum} % numbered lists where each item can be labeled
\usepackage{cite} % for grouping references in the text which are adjacent in the bibliography
\usepackage[nameinlink,capitalize,noabbrev]{cleveref}
\usepackage{nicefrac}

\usepackage{tikz-cd} 

\usepackage[font=small,labelfont=bf]{caption}

\AtBeginEnvironment{quote}{\small}

\newcommand{\NN}{\mathbb{N}}    % natural numbers            
\newcommand{\RR}{\mathbb{R}}    % real numbers                      
\newcommand{\F}{\mathcal{F}}    % face lattice
\newcommand{\Sph}{\mathbf{S}}   % sphere

\newcommand{\Tsymb}{\top}
\newcommand{\T}{^{\Tsymb}}

\newcommand{\nozero}{\setminus\{0\}}

\def\^#1{^{(#1)}}
\def\s^#1{^{\smash{(#1)}}}

\def\:{\colon}

\newcommand{\cupdot}{\mathbin{\mathaccent\cdot\cup}}

\newcommand{\labelstyle}[1]{\upshape(\textit{#1})}
\newcommand{\mylabel}{\labelstyle{\roman*}}
\newenvironment{myenumerate}{\begin{enumerate}[label=\mylabel]}{\end{enumerate}}

\def\itm#1{\labelstyle{\romannumeral#1\relax}}

\makeatletter
\newcommand{\myitem}[1]{%
\item[#1]\protected@edef\@currentlabel{#1}%
}
\makeatother

\newcommand{\freespace}{\kern.07em} 
\newcommand{\free}{\freespace\cdot\freespace} 

\newcommand{\enquote}[1]{``#1''}                                 % quotation marks

\newcommand{\ul}[1]{\underline{\smash{#1}}}

\newcommand{\msays}[1]{{\footnotesize\textcolor{red}{#1}}}
\newcommand{\TODO}{\msays{TODO}}

\theoremstyle{plain}  % theorem, lemma, corollary, proposition, conjecture, criterion, algorithm
\newtheorem{theorem}{Theorem}[section]
\newtheorem{corollary}[theorem]{Corollary}
\newtheorem{lemma}[theorem]{Lemma}
\newtheorem{proposition}[theorem]{Proposition}

\theoremstyle{definition} % definition, condition, problem, example
\newtheorem{definition}[theorem]{Definition}
\newtheorem{example}[theorem]{Example}
\newtheorem{remark}[theorem]{Remark}
\newtheorem{question}[theorem]{Question}
\newtheorem{observation}[theorem]{Observation}

\crefname{theorem}{Theorem}{Theorems}
\crefname{proposition}{Proposition}{Propositions}
\crefname{lemma}{Lemma}{Lemmas}
\crefname{corollary}{Corollary}{Corollaries}
\crefname{remark}{Remark}{Remarks}
\crefname{example}{Example}{Examples}
\crefname{definition}{Definition}{Definitions}
\crefname{problem}{Problem}{Problems}
\crefname{observation}{Observation}{Observation}
\crefname{construction}{Construction}{Construction}

\DeclareMathOperator{\aff}{aff}
\DeclareMathOperator{\conv}{conv}
\DeclareMathOperator{\cone}{cone}
\DeclareMathOperator{\Aut}{Aut}

\DeclareMathOperator{\Ortho}{O}

\DeclareMathOperator{\Span}{span}

\DeclareMathOperator{\Orb}{Orb}   		% orbit polytope

\let\eps=\epsilon
\let\eset=\varnothing

\let\<=\langle
\let\>=\rangle
\let\x=\times

\def\...{...}
\newcommand{\shortStyle}{\textit}
\newcommand{\ie}{\shortStyle{i.e.,}}

\newcommand{\eg}{\shortStyle{e.g.}}

\newcommand{\wrt}{\shortStyle{w.r.t.}}

\newcommand{\cf}{\shortStyle{cf.}}

\newcommand{\resp}{resp.}

\let\angle=\measuredangle
\makeatletter
\renewcommand*{\eqref}[1]{%
  \hyperref[{#1}]{\textup{\tagform@{\ref*{#1}}}}%
}
\makeatother

\numberwithin{equation}{section}

\begin{document}

%%%%%%%%%%%%%%%%%%%%%%%%%%%%%%%%%%%%%%%%%%%%%%%%%%%%%%%%%%%%%%%%%%%%%%%%%%%%%%%%%%%%%%

\expandafter\title
%[On polytopes that are edge-transitive but not vertex-transitive]
%{On polytopes that are edge-transitive but not vertex-transitive}
{The edge-transitive polytopes that are not vertex-transitive}
		
\author[F. Göring]{Frank Göring}
\author[M. Winter]{Martin Winter}
\address{Faculty of Mathematics, University of Technology, 09107 Chemnitz, Germany}
\email{frank.goering@mathematik.tu-chemnitz.de}
\address{Mathematics Institute, University of Warwick, Coventry CV4 7AL, United Kingdom}
\email{martin.h.winter@warwick.ac.uk}
	
		
\date{\today}
\begin{abstract}
In 3-dimensional Euclidean space there exist two exceptional polyhedra, the \emph{rhombic dodecahedron} and the \emph{rhombic triacontahedron}, the only known polytopes (besides polygons) that are edge-transitive without being vertex-transitive.
We show that these polyhedra do not have higher-dimensio\-nal analogues, that is, that
in dimension $d\ge 4$, edge-transitivity of convex polytopes implies vertex-transitivity.

More generally, we give a classification of all convex polytopes which at the same time have all edges of the same length, an edge in-sphere and a bipartite edge-graph.
We show that any such polytope in dimension $d\ge 4$ is vertex-transitive.
%Since any edge- but not vertex-transitive polytope is of this form, the claim follows.

%As an application, we classify the edge-transitive zonotopes, as well as the 4-dimensional edge-transitive polytopes.

%\TODO

%We prove that in $d\ge 4$ dimensions every edge-transitive \emph{convex} polytope is also vertex-transitive, that is, that there are no \emph{purely edge-transitive} polytopes (edge- but not vertex-transitive) except for the well-known examples in dimension $d\in\{2,3\}$.
%
%For this, we generalizes being purely edge-transitive to the notion of being \emph{edge-regular}: a polytope is edge-regular, if
%%
%\begin{myenumerate}
%	\item all edges are of the same length,
%	\item the edge graph $G=(V_1\cupdot V_2,E)$ is bipartite, and
%	\item there are $r_1\le r_2$ so that $\|v\|=r_i$ for all $v\in V_i$.
%\end{myenumerate}
%We show that this class actually coincides with the purely edge-transitive polytopes plus the vertex-transitive zonotopes.
%
%This allows us to give a classification of edge-transitive zonotopes.
\end{abstract}

\maketitle

%\setcounter{tocdepth}{1}
%\tableofcontents

\section{Introduction}
\label{sec:introduction}

%\TODO\msays{: replace spherical polyhedron with tetrahedra. 1. show that tetrahedra are disjoint, 2. show that $\beta>\alpha$ always.}

A $d$-dimensional (convex) polytope $P\subset\RR^d$ is the convex hull of finitely many points.
%For our purpose, $P$ will always be of full dimension, that is, not contained in a proper affine subspace.
The polytope $P$ is \emph{vertex-transitive} \resp\ \emph{edge-transitive} if its (orthogonal) symmetry group $\Aut(P)\subset\Ortho(\RR^d)$ acts transitively on its vertices \resp\ edges.

It has long been known that there are exactly \emph{nine} edge-transitive polyhedra in $\RR^3$ (see \eg\ \cite{grunbaum1987edge}).
These are the five Platonic solids (tetrahedron, cube, octahedron, icosahedron and \mbox{dodecahedron}) together with the cuboctahedron, the icosidodecahedron, and their duals, the \emph{rhombic dodecahedron} and the \emph{rhombic triacontahedron} (depicted below in this roder): %(these are depicted below, in this order; the latter two will be important in the following).
\begin{center}
\includegraphics[width=0.8\textwidth]{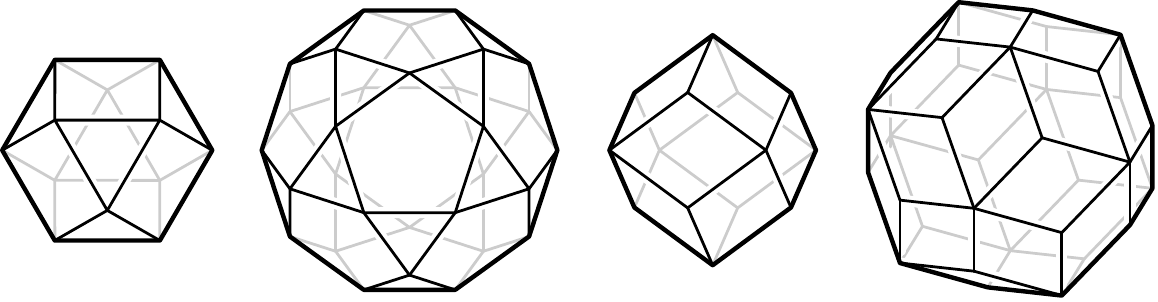}
\end{center}
Little is known about the analogous question in higher dimensions.
Branko Grünbaum writes in \enquote{Convex Polytopes} \cite[p.\ 413]{grunbaum2013convex}

\begin{quote}
No serious consideration seems to have been given to polytopes in dimension $d\ge 4$ about which transitivity of the symmetry group is assumed only for faces of suitably low dimensions, [...].
\end{quote}

Even though families of higher-dimensional edge-transitive polytopes have been studied, to the best of our knowledge, no classification of these has been achieved so far.
%
%Those edge-transitive polytopes which are known in four and more dimensions are all certain modifications of regular polytopes, and products thereof.
%For example, given a regular $n$-gon $P$, the cartesian product $P\times P\subset\RR^4$ is an edge-transitive 4-polytope, the so-called $(n,n)$-duoprism.
%Some examples of edge-transitive polytopes in $d\ge 4$ dimensions are known.~Besides the higher-dimensional regular polytopes, one can construct edge-transitive polytopes as products:
%For example, if $P\subset\RR^d$ is a regular polytope, then $P\times\cdots\times P\subset\RR^{md}$ (the $m$-fold cartesian product of $P$ with itself) is edge-transitive. 
%For example, choosing $P$ to be a regular $n$-gon, $P\times P$ yields the so-called \emph{$(n,n)$-duoprism}.
%Other modifications of regular polytopes are known to be edge-transitive as well.
%Nevertheless, to our best knowledge, no classification of these has been achieved so far.
%
Equally striking, all the known examples of such polytopes in dimension at least four are simultaneously \emph{vertex-transitive}.
%All the higher-dimensional edge-transitive polytopes mentioned above, and all further example known, turned out to be \emph{vertex-transitive} as well.
% (\ie\ the symmetry group acts transitively on the vertices).
In dimension up to three, certain polygons (see \cref{fig:2n_gons}), as well as the rhombic dodecahedron and rhombic triacontahedron are edge- but \emph{not} vertex-transitive.
No higher dimensional example of this kind has been found.
In this paper we prove that this is not for lack of trying: 

\begin{theorem}\label{res:edge_implies_vertex}
In dimension $d\ge 4$, edge-transitivity of convex polytopes implies vertex-transitivity.
\end{theorem}

As immediate consequence, we obtain the classification of all polytopes that are edge- but not vertex-transitive.
The list is quite short:

%\begin{figure}
%\centering
%\includegraphics[width=0.5\textwidth]{img/edge_no_vertex_transitive}
%\end{figure}

\begin{figure}
\centering
\includegraphics[width=0.75\textwidth]{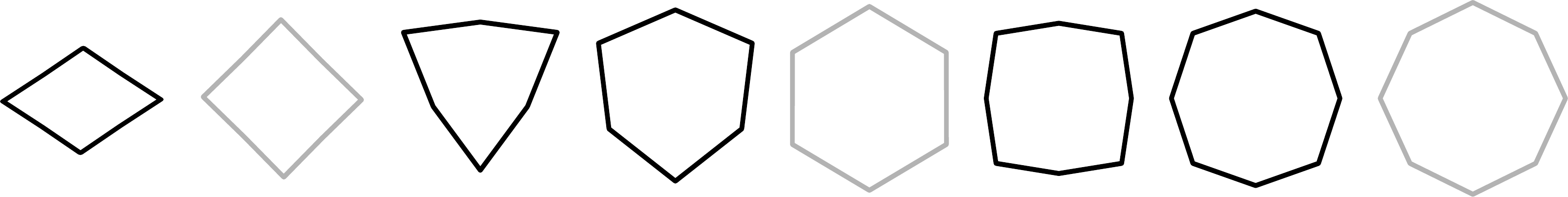}
\caption{Some examples of edge-transitive $2n$-gons with $2n\in\{4,6,8\}$ (the same works for all $n$). The polygons depicted with black boundary are not vertex-transitive.}
\label{fig:2n_gons}
\end{figure}

\begin{corollary}
If $P\subset\RR^d,d\ge 2$ is edge- but not vertex-transitive, then $P$ is one of the following:
\begin{myenumerate}
\item a non-regular $2k$-gon (see \cref{fig:2n_gons}),
\item the rhombic dodecahedron, or
\item the rhombic triacontahedron.
\end{myenumerate}
\end{corollary}

\Cref{res:edge_implies_vertex} is proven by embedding the class of edge- but not vertex-transitive polytopes in a larger class of polytopes, defined by geometric regularities instead of symmetry.
In \cref{res:trans_is_bipartite} we show that a polytope $P\subset\RR^d$ which is edge- but not vertex-transitive must have all of the following properties:
\begin{myenumerate}
	\item all edges are of the same length,
	\item it has a bipartite edge-graph $G_P=(V_1\cupdot V_2,E)$, and
	%\item $P$ has an edge insphere, that is, there is a sphere touching each edge in exactly one point,
	\item there are radii $r_1\le r_2$, so that $\|v\|=r_i$ for all $v\in V_i$. % (given $(i)$, this is equivalent to having an edge in-sphere, see \cref{rem:alternative_definiton}).
\end{myenumerate}
We compile this into a definition: a polytope that has these three properties shall be called \emph{bipartite} (\cf\ \cref{def:bipartite}).
The edge- but not vertex-transitive polytopes then form a subclass of the bipartite polytopes, but the class of bipartite polytopes is much better behaved.
For example, faces of bipartite polytopes are bipartite (\cref{res:faces_of_bipartite}), something which is not true for edge/vertex-transitive polytopes\footnote{For example, consider a vertex-transitive but not uniform antiprism. Its faces are non-regular triangles, which are thus not vertex-transitive. Alternatively, consider the $(n,n)$-duoprism, $n\not=4$, that is, the cartesian product of a regular $n$-gon with itself. This polytope is edge-transitive, but its facets are $n$-gonal prisms (the cartesian product of a regular $n$-gon with an edge), which are not edge-transitive.}.
\mbox{Our quest is then to classify all bipartite polytopes.
The~sur}\-prising result: already being bipartite is very restrictive:

\begin{theorem}
If $P\subset\RR^d,d\ge 2$ is bipartite, then $P$ is one of the following:
\begin{myenumerate}
\item an edge-transitive $2k$-gon (see \cref{fig:2n_gons}),
\item the rhombic dodecahedron,
\item the rhombic triacontahedron, or 
\item a $\Gamma$-permutahedron for some finite reflection group $\Gamma\subset\Ortho(\RR^d)$ (see \cref{def:permutahedron}; some 3-dimensional examples are shown in \cref{fig:permutahedron}).
\end{myenumerate}
\end{theorem}

\begin{figure}[h!]
\centering
\includegraphics[width=0.6\textwidth]{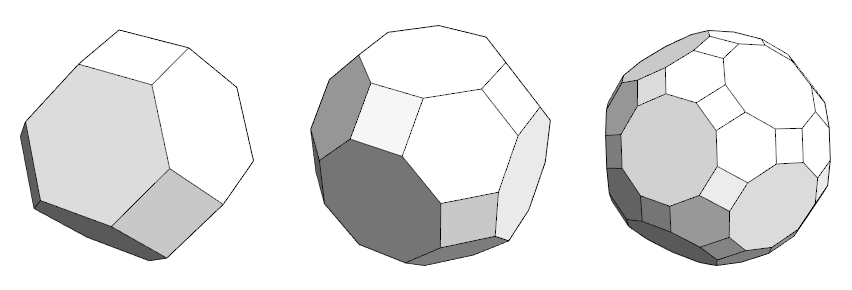}
\caption{From left to right: the $A_3$-, $B_3$ and $H_3$-permutahedron.}
\label{fig:permutahedron}
\end{figure}

%A $\Gamma$-permutahedron is a generalization of the standard permutahedron to general reflection groups $\Gamma$ (see \cref{def:permutahedron}).
%The $\Gamma$-permutahedra are all vertex-transitive,
%Since all of the generalized permutahedra are vertex-transitive, and these are the only higher-dimensional bipartite polytopes, we immediately obtain \cref{res:edge_implies_vertex}.
The $\Gamma$-permutahedra are vertex-transitive, and all the other entries in the list are of dimension $d\le 3$. This immediately implies \cref{res:edge_implies_vertex}.

Remarkably, despite the definition of bipartite polytope being purely geometric, all bipartite polytopes are highly symmetric, that is, at least vertex- or facet-transitive, and sometimes even edge-transitive.

\subsection{Overview}

In \cref{sec:bipartite_polytopes} we introduce the central notion of \emph{bipartite \mbox{polytope}} and prove its most relevant properties: that being bipartite generalizes being edge- but not vertex-transitive, and that all faces of bipartite polytopes are again bipartite.
We then investigate certain subclasses of bipartite polytopes: bipartite polygons and inscribed bipartite polytopes.
We prove that the latter coincide with the $\Gamma$-permutahedra, a class of vertex-transitive polytopes.
It therefore remains to classify the non-inscribed cases, the so-called \emph{strictly} bipartite polytopes.
We show that the classification of these reduces to the classification of bipartite \emph{polyhedra}, \ie\ the case $d=3$. 

From \cref{sec:bipartite_polyhedra} on the investigation is focused on the class of strictly bipartite polyhedra.
We successively determine restrictions on the structure of such, \eg\ the degrees of their vertices and the shapes of their faces.
This quite elaborate process uses many classical geometric results and techniques, including spherical polyhedra, the classification of rhombic isohedra and the realization of graphs as edge-graphs of polyhedra.
As a result, we can exclude all but two cases, namely, the rhombic dodecahedron, and the rhombic triacontahedron.
Additionally, we shall find a remarkable near-miss, that is, a polyhedron which fails to be bipartite only by a tiny (but quantifiable) amount.

\section{Bipartite polytopes}
\label{sec:bipartite_polytopes}

From this section on let $P\subset\RR^d,d\ge 2$ denote a $d$-dimensional polytope of full dimension (\ie\ $P$ is not contained in a proper affine subspace).
By $\F(P)$ we denote the face lattice of $P$, and by $\F_\delta(P)\subset\F(P)$ the subset of $\delta$-dimensional faces.

\begin{definition}\label{def:bipartite}
$P$ is called \emph{bipartite}, if
\begin{myenumerate}
	\item all its edges are of the same length $\ell$,
	\item its edge-graph is bipartite, which we write as $G_P=(V_1\cupdot V_2,E)$, and
	
	\item there are radii $r_1\le r_2$ so that $\|v\|=r_i$ for all $v\in V_i$.
\end{myenumerate}
If $r_1<r_2$, then $P$ is called \emph{strictly bipartite}.
A vertex $v\in V_i$ is called an \emph{$i$-vertex}.
The numbers $r_1$, $r_2$ and $\ell$ are called the \emph{parameters} of a bipartite polytope.
\end{definition}

\begin{remark}
\label{rem:abuse_of_notation}
Since $P$ is full-dimensional by convention, \cref{def:bipartite} only defines \emph{full-dimensional} bipartite polytopes.
%So far we defined \emph{full-dimensional} bipartite polytopes ($P$ denotes a full-dimensional polytope by convention).
%This has the advantage that there is a unique placement of $P$ (relative to the origin) that makes it bipartite, and so the the radii $r_1$ and $r_2$ are already uniquely determined by the shape of $P$ without considering its placement.
%
%In contrast, if $P$ where not full-dimensional but still bipartite, then translating it orthogonally to its affine hull will keep it bipartite but change the radii.

To extend this notion to not necessarily full-dimensional polytopes, we shall call a polytope \emph{bipartite} even if it is just bipartite as a subset of its affine hull where we made an appropriate choice of origin in the affine hull (note that whether a polytope is bipartite depends on its placement relative to the origin and that there is at most one such placement if the polytope is full-dimensional).
This comes in handy when we discuss faces of bipartite polytopes. % (\eg\ in \cref{res:faces_of_bipartite}).

%By a slight abuse of notation, and in order to obtain parameters that only depend on the shape, a not necessarily full-dimensional polytope is called bipartite if it is bipartite as a subset of its affine hull (with an appropriate choice of origin inside the affine hull).
%%However, for convenience and by a slight abuse of notation we shall call a polytope also bipartite already if it is bipartite as a subset of its affine hull (where we consider the origin placed inside the affine hull at the unique point that makes the polytope bipartite).
%This will come in handy when we discuss faces of bipartite polytopes (\eg\ in \cref{res:faces_of_bipartite}).
\end{remark}

\begin{remark}
\label{rem:alternative_definiton}
An alternative definition of bipartite polytope would replace \itm3 by the condition that $P$ has an \emph{edge in-sphere}, that is, a sphere that touches each edge of $P$ in a single point (this definition was used in the abstract). %\footnote{When $d=3$, this is also known as a \emph{midsphere}.}.
\mbox{The configuration~de}\-picted below (an edge of $P$ connecting two vertices $v_1\in V_1$ and $v_2 \in V_2$) shows~how any one of the four quantities $r_1,r_2,\ell$ and $\rho$ (the radius of the edge in-sphere) is determined from the other three by solving the given set of equations:
\begin{center}
\includegraphics[width=0.35\textwidth]{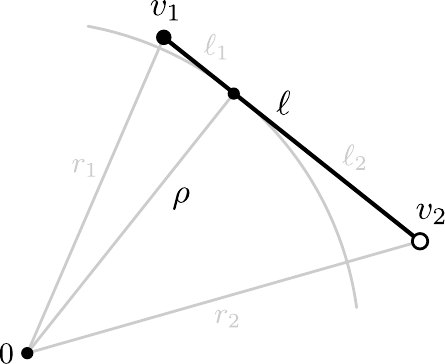}
\qquad\quad\raisebox{5em}{$\displaystyle
\begin{array}{rcl}
\rho^2+\ell_1^2&\!\!\!\!=\!\!\!\!\!&r_1^2\\[0.3ex]
\rho^2+\ell_2^2&\!\!\!\!=\!\!\!\!\!&r_2^2\\[0.3ex]
\ell_1+\ell_2&\!\!\!\!=\!\!\!\!\!&\ell
\end{array}
$}
\end{center}

There is a subtlety: for the edge in-sphere to actually touch the edge (rather than only its affine hull outside of the edge) it is necessary that the perpendicular projection of the origin onto the edge ends up inside the edge (equivalently, that the triangle $\conv\{0,v_1,v_2\}$ is acute at $v_1$ and $v_2$). One might regard this as intuitively clear since we are working with convex polytopes, but this will also follows formally as part of our proof of \cref{res:0_in_P} (as we shall mention there in a footnote).

This alternative characterization of bipartite polytopes via edge in-spheres will become relevant towards the end of the classification (in \cref{sec:446}).
Still, for the larger part of our investigation, \cref{def:bipartite} \itm3 is the more convenient version to work with.
\end{remark}

\subsection{General obsevations}

\begin{proposition}\label{res:trans_is_bipartite}
If $P$ is edge- but not vertex-transitive, then $P$ is bipartite.
\end{proposition}

This is a geometric analogue to the well known fact that every edge- but not vertex-transitive \emph{graph} is bipartite.
A proof of the graph version can be found in \cite{godsil1978graphs}.
The following proof can be seen as a geometric analogue:

\begin{proof}[Proof of \cref{res:trans_is_bipartite}]
Clearly, all edges of $P$ are of the same length.

Fix some edge $e\in\F_1(P)$ with end vertices $v_1,v_2\in\F_0(P)$.
Let $V_i$ be the orbit of $v_i$ under $\Aut(P)$.
We prove that $V_1\cup V_2=\F_0(P)$, $V_1\cap V_2=\eset$ and that the edge graph $G_P$ is bipartite with partition $V_1\cupdot V_2$.

Let $v\in \F_0(P)$ be some vertex and $\tilde e\in\F_1(P)$ an incident edge.
By edge-transitivity, there is a symmetry $T\in\Aut(P)$ that maps $\tilde e$ onto $e$, and therefore maps $v$ onto $v_i$ for some $i\in\{1,2\}$.
Thus, $v$ is in the orbit $V_i$.
This holds for all vertices of $P$, and therefore $V_1\cup V_2=\F_0(P)$.

The orbits of $v_1$ and $v_2$ must either be identical or disjoint.
Since $V_1\cup V_2=\F_0(P)$, from $V_1=V_2$ it would follow $V_1=\F_0(P)$, stating that $P$ has a single orbit of vertices.
But since $P$ is \emph{not} vertex-transitive, this cannot be.
Thus, $V_1\cap V_2=\eset$, and therefore $V_1\cupdot V_2=\F_0(P)$.

Let $\tilde e\in\F_1(P)$ be an edge with end vertices $\tilde v_1$ and $\tilde v_2$.
By edge-transitivity, $\tilde e$ can be mapped onto $e$ by some symmetry $T\in\Aut(P)$.
Equivalently $\{T\tilde v_1,T\tilde v_2\}=\{v_1,v_2\}$.
Since $v_1$ and $v_2$ belong to different orbits under $\Aut(P)$, so do $\tilde v_1$ and $\tilde v_2$.
Hence $\tilde e$ has one end vertex in $V_1$ and one end vertex in $V_2$.
This holds for all edges, and thus, $G_P$ is bipartite with partition $V_1\cupdot V_2$.

It remains to determine the radii $r_1\le r_2$.
Set $r_i:=\|v_i\|$ (assuming w.l.o.g.\ that $\|v_1\|\le\|v_2\|$). 
Then for every $v\in V_i$ there is a symmetry $T\in\Aut(P)\subset\Ortho(\RR^d)$ so that $Tv_i=v$, and thus
$$\|v\|=\|T v_i\| = \|v_i\|=r_i.$$
%
%This proves part $(iii)$ of \cref{def:bipartite}.
\end{proof}

Bipartite polytopes are more comfortable to work with than edge- but not vertex-transitive polytopes because their faces are again bipartite polytopes (in the sense as explained in \cref{rem:abuse_of_notation}). 
Later, this will enable us to reduce the problem to an investigation in lower dimensions.
%We previously same is not true for vertex- or edge-transitive polytopes

\begin{proposition}\label{res:faces_of_bipartite}
Let $\sigma\in\F(P)$ be a face of $P$. Then it holds
\begin{myenumerate}
	\item if $P$ is bipartite, so is $\sigma$.
	\item if $P$ is strictly bipartite, then so is $\sigma$, and $v\in\F_0(\sigma)\subseteq\F_0(P)$ is an $i$-vertex in $P$ if and only if it is an $i$-vertex in $\sigma$.
	\item if $r_1\le r_2$ are the radii of $P$ and $\rho_1\le \rho_2$ are the radii of $\sigma$, then there holds
	$$h^2 + \rho_i^2 = r_i^2,$$
	where $h$ is the height of $\sigma$, that is, the distance of $\aff(\sigma)$ from the origin.
\end{myenumerate}
%
%By a slight abuse of notation, in $(i)$ and $(ii)$ we mean that $\sigma$ is bipartite when considered around the center $c$, where $c$ is the orthogonal projection of the origin onto $\aff(\sigma)$.
%
\begin{proof}

Properties clearly inherited by $\sigma$ are that all edges are of the same length and that the edge graph is bipartite.
It remains to show the existence of the radii $\rho_1\le \rho_2$ compatible with the bipartition of the edge-graph of $\sigma$.

Let $c\in\aff(\sigma)$ be the orthogonal projection of $0$ onto $\aff(\sigma)$.
Then $\|c\|=h$, the height of $\sigma$ as defined in $(iii)$.
For any vertex $v\in\F_0(\sigma)$ which is an $i$-vertex in $P$, the triangle $\Delta:=\conv\{0,c,v\}$ has a right angle at $c$.
Set $\rho_i:=\|v-c\|$ and observe
\begin{equation}
\quad \rho_i^2:=\|v-c\|^2 = \|v\|^2-\|c\|^2 = r_i^2-h^2. 
\tag{$*$}
\end{equation}
In particular, the value $\rho_i$ does only depend on $i$.
In other words, $\sigma$ is a bipartite polytope when considered as a subset of its affine hull, where the origin is chosen to be $c$ (\cf\  \cref{rem:abuse_of_notation}).
This proves $(i)$, and $(*)$ is equivalent to the equation in $(iii)$.
From $(*)$ also follows $r_1< r_2\Leftrightarrow \rho_1< \rho_2$, which proves $(ii)$.
\end{proof}
\end{proposition}

\iftrue

%We start with an observation~about angles between vertices (stated for general bipartite polytopes).

The following observation will be of use later on.

\begin{observation}\label{res:angles_between_vertices}
Given two adjacent vertices $v_1,v_2\in\F_0(P)$ with $v_i\in V_i$, and if $P$ has parameters $r_1,r_2$ and $\ell$, then
$$\ell^2 = \|v_1-v_2\|^2 = \|v_1\|^2+\|v_2\|^2 - 2\<v_1,v_2\> = r_1^2+r_2^2-2 r_1 r_2 \cos\angle(v_1,v_2),$$ 
%
%which rearranges to
%%
%$$\angle(v_1,v_2) = \arccos\Big(\frac{r_1^2+r_2^2-\ell^2}{2r_1 r_2}\Big).$$
%
This can be rearranged for $\cos\angle(v_1,v_2)$. While the exact value  of this expression is not of relevance to us, this shows that this angle is determined by the parameters and does not depend on the choice of the adjacent vertices $v_1$ and $v_2$.

%In the case of a bipartite $n$-gon, where the edges form a closed cycle, this implies that the vectors of the vertices are equally spread out by an angle of $2\pi/n$.
%That is, for any two adjacent vertices $v_1,v_2\in\F_0(P)$ holds 
%%
%$$\angle(v_1,v_2)=\frac{2\pi}n.$$
\end{observation}

\subsection{Bipartite polygons}

The easiest to describe (and to explicitly construct) are the bipartite \emph{polygons}.
%This section lists their most relevant properties for use in later sections.

Foremost, the edge-graph is bipartite, and thus, a bipartite polygon must be a $2k$-gon for some $k\ge 2$.
One can show that the bipartite polygons are exactly the edge-transitive $2k$-gons (\cf\ \cref{fig:2n_gons}), and that such one is \emph{strictly} bipartite if and only if it is \emph{not} vertex-transitive (or equivalently, not regular).
%Proving this is an easy exercise and is left to the reader.
We will not make use of these symmetry properties of bipartite polygons.
%Instead, we are satisfied with determining those properties of bipartite polygons that are most relevant to the study of higher-dimensional bipartite polytopes.

The parameters $r_1, r_2$ and $\ell$ uniquely determine a bipartite polygon, as can be seen by explicit construction:
\begin{center}
\includegraphics[width=0.98\textwidth]{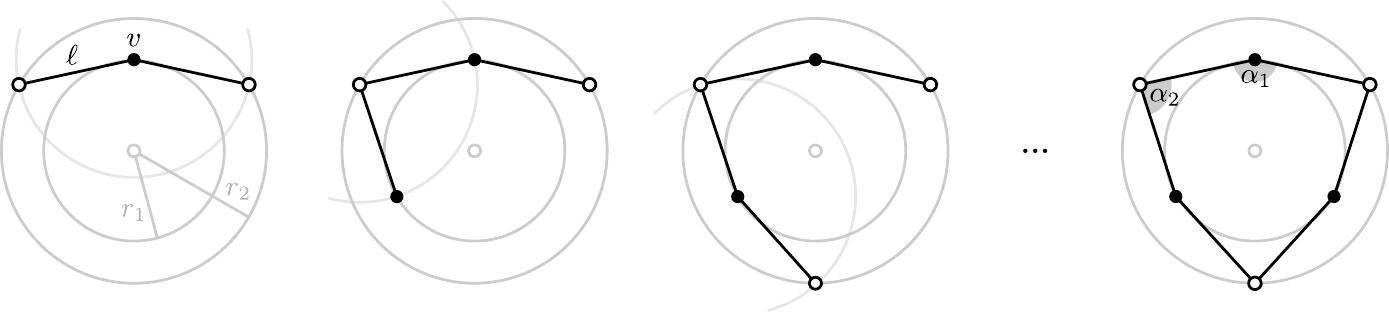}
\end{center}
One starts with an arbitrary 1-vertex $v\in\RR^2$ placed on the circle $\Sph_{r_1}(0)$.
Its~neigh\-boring vertices are then uniquely determined as the intersections $\Sph_{r_2}(0)\cap\Sph_\ell(v)$.
The procedure is repeated with the new vertices until the edge cycle closes (which only happens if the parameters are chosen appropriately).

The procedure also makes clear that the interior angle $\alpha_i\in(0,\pi)$ at an $i$-vertex only depends on $i$, but not on the chosen vertex $v\in V_i$.
%The exact value of these angles will not be important, but just that they exist.

%We summarize

\begin{corollary}\label{res:2k_gons}
A bipartite polygon $P\subset\RR^2$ is a $2k$-gon with alternating interior angles $\alpha_1,\alpha_2\in(0,\pi)$ ($\alpha_i$ being the interior angle at an $i$-vertex), and its shape is uniquely determined by its parameters (up to congruence).
\end{corollary}

The exact values for the interior angles are not of relevance. % (if necessary, they can be obtained by a similar computation as the angle in \cref{res:angles_between_vertices}).
Instead, we only need the following properties:

\begin{proposition}\label{res:angles}
The interior angles $\alpha_1,\alpha_2\in(0,\pi)$ satisfy
%
%\begin{equation}\label{eq:angles}
%\alpha_2\le \alpha_{\mathrm{reg}}^k\le\alpha_1, \quad\text{with $\alpha_{\mathrm{reg}}^k:=\Big(1-\frac1k\Big)\pi$}.
%\end{equation}
%%
%The angle $\alpha_{\mathrm{reg}}^k$ is the interior angle of a regular $2k$-gon.
%Equality holds in either part of \eqref{eq:angles} if and only if $r_1=r_2$.
\begin{equation}\label{eq:angles}
\alpha_1+\alpha_2=2\alpha_{\mathrm{reg}}^k
\quad\text{and}\quad
\alpha_2\le \alpha_{\mathrm{reg}}^k\le\alpha_1,
\end{equation}
where $\alpha_{\mathrm{reg}}^k:=(1-1/k)\pi$ is the interior angle of a regular $2k$-gon, and the inequalities are satisfied with equality if and only  $r_1=r_2$.
\begin{proof}
The sum of interior angles of a $2k$-gon is $2(k-1)\pi$, and thus
$k \alpha_1+k\alpha_2=2(k-1)\pi$, which, after division by $k$, yields the first part of \eqref{eq:angles}.

For two adjacent vertices $v_1,v_2\in \F_0(P)$ (where $v_i\in V_i$), consider the triangle $\Delta:=\conv\{0,v_1,v_2\}$ whose edge lengths are $r_1, r_2$ and $\ell$, and whose interior angles at $v_1$ \resp\ $v_2$ are $\alpha_1/2$ \resp\ $\alpha_2/2$.
From $r_1\le r_2$ (\resp\ $r_1<r_2$) and the law of sine follows $\alpha_1 \ge \alpha_2$ (\resp\ $\alpha_1 > \alpha_2$).
With $\alpha_1+\alpha_2=2\alpha_{\mathrm{reg}}^k$ this yields the second part of \eqref{eq:angles}.
\end{proof}
\end{proposition}

\begin{observation}\label{res:2k_4_case}
For later use (in \cref{res:strictly_bipartite_444}), consider \cref{res:angles} with $2k=$ $4$.
In this case we find,
$$\alpha_2\le \frac\pi 2\le \alpha_1,$$
that is, $\alpha_1$ is never acute, and $\alpha_2$ is never obtuse.
\end{observation}

\subsection{The case $r_1=r_2$}

We classify the inscribed bipartite polytopes, that is, those with coinciding radii $r_1=r_2$.
This case is made especially easy by a classification result from \cite{winter2019classification}.
We need the following definition:

\begin{definition}\label{def:permutahedron}
Let $\Gamma\subset\Ortho(\RR^d)$ be a finite reflection group and $v\in \RR^d$ a \emph{generic} point \wrt\ $\Gamma$ (\ie\ $v$ is not fixed by a non-identity element of $\Gamma$).
The orbit~polytope
$$\Orb(\Gamma,v):=\conv\{Tv\mid T\in\Gamma\}\subset\RR^d$$
is called a \emph{$\Gamma$-permutahedron}.
\end{definition}

The relevant result then reads

\begin{theorem}[Corollary 4.6.\ in \cite{winter2019classification}]
\label{res:classification_inscribed_zonotopes}
If $P$ has only centrally symmetric 2-dimen\-sional faces (that is, it is a zonotope), has all vertices on a common sphere and all edges of the same length, then $P$ is a $\Gamma$-permutahedron.
\end{theorem}

This provides a classification of bipartite polytopes with $r_1=r_2$.

\begin{theorem}
If $P\subset\RR^d$ is bipartite with $r_1=r_2$, then it is a $\Gamma$-permutahedron.
\begin{proof}
If $r_1=r_2$, then all vertices are on a common sphere (that is, $P$ is inscribed).
By definition, all edges are of the same length.
Both statements then also hold for the faces of $P$, in particular, the 2-dimensional faces. %holds for all faces of $P$ too, in particular, its 2-faces.
An inscribed polygon with a unique edge length is necessarily regular.
With \cref{res:2k_gons} the 2-faces are then regular $2k$-gons, therefore centrally symmetric.

Summarizing, $P$ is inscribed, has all edges of the same length, and all 2-dimen\-sio\-nal faces of $P$ are centrally symmetric.
By \cref{res:classification_inscribed_zonotopes}, $P$ is a $\Gamma$-permutahedron. % for some finite reflection group $\Gamma\subset\Ortho(\RR^d)$.
\end{proof}
\end{theorem}

%A $\Gamma$-permutahedron is not necessarily bipartite. %, since not all edges must be of the same length.
%But for every finite reflection group $\Gamma\subset\Ortho(\RR^d)$ there is a generic point $v\in\RR^d$ so that all edges of $\Orb(\Gamma,v)$ have the same length, which is enough to make it bipartite.
%
%Nevertheless,
$\Gamma$-permutahedra are vertex-transitive by definition, hence do not provide examples of edge- but not vertex-transitive polytopes.

\subsection{Strictly bipartite polytopes}

It remains to classify the \emph{strictly} bipartite~poly\-topes.
This problem is divided into two independent cases: dimension $d=3$, and dimension $d\ge 4$.
The detailed study of the case $d=3$ (which turns out to be the actual hard work) is postponed until \cref{sec:bipartite_polyhedra}, the result of which is the following theorem:

%In \cref{res:3_dim_suffices}, we show that this problem reduces to the classification of strictly bipartite \emph{polyhedra}, that is, to the case $d=3$.
%The classification of the bipartite polyhedra turns out to be the actual hard work, and is postpones until \cref{sec:bipartite_polyhedra}.
%The final result of this is the following theorem:

\begin{theorem}\label{res:strictly_bipartite_polyhedra}
If $P\subset\RR^3$ is strictly bipartite, then $P$ is the rhombic dodecahedron or the rhombic triacontahedron.
\end{theorem}

Presupposing \cref{res:strictly_bipartite_polyhedra}, the case $d\ge 4$ is done quickly.

\begin{theorem}\label{res:3_dim_suffices}
There are no strictly bipartite polytopes in dimension $d\ge 4$.
\begin{proof}
It suffices to show that there are no strictly bipartite polytopes in dimension $d=4$, as any higher-dimensional example has a strictly bipartite 4-face (by \cref{res:faces_of_bipartite}).

Let $P\subset\RR^4$ be a strictly bipartite 4-polytope.
Let $e\in\F_1(P)$ be an edge of $P$.
Then there are $s\ge 3$ cells (aka.\ 3-faces) $\sigma_1,...,\sigma_s\in\F_3(P)$ incident to $e$, each of which is again strictly bipartite (by \cref{res:faces_of_bipartite}).
By \cref{res:strictly_bipartite_polyhedra} each $\sigma_i$ is a rhombic dodecahedron~or~rhom\-bic triacontahedron.

The dihedral angle of the rhombic dodecahedron \resp\ triacontahedron is $120^\circ$ \resp\ $144^\circ$ at every edge \cite{coxeter1973regular}.
However, the dihedral angles meeting at $e$ must sum up to less than $2\pi$.
With the given dihedral angles this is impossible.
\end{proof}
\end{theorem}

\section{Strictly bipartite polyhedra}
\label{sec:bipartite_polyhedra}

In this section we derive the classification of strictly bipartite polyhedra.
The main goal is to show that there are only two: the rhombic dodecahedron and the rhombic triacontahedron.

From this section on, let $P\subset\RR^3$ denote a fixed \emph{strictly bipartite polyhedron} with radii $r_1< r_2$ and edge length $\ell$.
The 2-faces of $P$ will be shortly referred to as just \emph{faces} of $P$.
Since they are bipartite, they are necessarily $2k$-gons.

%The following terminology comes in handy:

\begin{definition}
We use the following terminology:
\begin{myenumerate}
	\item a face of $P$ is of \emph{type} $2k$ (or called a \emph{$2k$-face}) if it is a $2k$-gonal polygon.
	\item an edge of $P$ is of \emph{type} $(2k_1,2k_2)$ (or called a \emph{$(2k_1,2k_2)$-edge}) if the two incident faces are of type $2k_1$ and $2k_2$ respectively.
	\item a vertex of $P$ is of \emph{type} $(2k_1,...,2k_s)$ (or called a \emph{$(2k_1,...,2k_s)$-vertex}) if its incident faces can be enumerated as $\sigma_1,...,\sigma_s$ so that $\sigma_i$ is a $2k_i$-face (note, the order of the numbers does not matter).
\end{myenumerate}
We write $\tau(v)$ for the type of a vertex $v\in\F_0(P)$.
\end{definition}

\subsection{General observations}

In a given bipartite polyhedron, the type of a vertex, edge or face already determines much of its metric properties.
We prove this for faces:

\begin{proposition}\label{res:unique_shape}
For some face $\sigma\in\F_2(P)$, any of the following properties of $\sigma$ determines the other two:
\begin{myenumerate}
	\item its type $2k$,
	\item its interior angles $\alpha_1>\alpha_2$.
	\item its height $h$ (that is, the distance of $\aff(\sigma)$ from the origin).
\end{myenumerate}
\end{proposition}

%A polygon is already uniquely described by stating its type, edge lengths and interior angles, and thus

%Consequently, any two faces of $P$ with the same height, or the same type, or the same interior angles are already congruent.
%This follows since two $n$-gons with the same same edge lengths and the same alternting interior angles must be congruent.

\begin{corollary}
Any two faces of $P$ of the same height, or the same type, or the same interior angles, are congruent.
\end{corollary}

\begin{proof}[Proof of \cref{res:unique_shape}]
Fix a face $\sigma\in\F_2(P)$. 
%Suppose $\sigma$ is of type $2k$, height $h$ and has interior angles $\alpha_1>\alpha_2$.

Suppose that the height $h$ of $\sigma$ is known.
By \cref{res:faces_of_bipartite}, a face of $P$ of height $h$ is bipartite with radii $\rho_i^2:=r_i^2-h^2$ and edge length $\ell$.
By \cref{res:2k_gons}, these parameters then uniquely determine the shape of $\sigma$, which includes its type and its interior angles. This shows $(iii)\implies(i),(ii)$.

\iffalse

Now suppose we know the interior angles $\alpha_1>\alpha_2$ of $\sigma$ (actually, it suffices to know one of these, say $\alpha_1$).
Fix a 1-vertex $v\in V_1$ of $\sigma$ and let $w_1,w_2\in V_2$ be its two adjacent 2-vertices in $\sigma$.
Then $\angle(w_1-v,w_2-v)=\alpha_1$.
All neighbors of $v$ lie on the circle $C:=\mathbf S_\ell(v)\cap\mathbf S_{r_2}(0)$, the symmetry axis of which is $\Span\{v\}$ (consider the figure below).
Thus, there is a unique way (up to orthogonal transformation) to place $w_1$ and $w_2$ on $C$ so that $\angle(w_1-v,w_2-v)=\alpha_1$.
But this uniquely determines the distance of $\aff(\sigma)=\aff\{v,w_1,w_2\}$ from the origin (the height of $\sigma$).
This shows $(ii)\implies(iii)$.

\begin{center}
\includegraphics[width=0.45\textwidth]{img/intersecting_spheres}
\end{center}

\hrulefill

\fi

Suppose now that we know the interior angles $\alpha_1>\alpha_2$ of $\sigma$ (it actually suffices to know one of these, say $\alpha_1$).
Fix a 1-vertex $v\in V_1$ of $\sigma$ and let $w_1,w_2\in V_2$ be its two adjacent 2-vertices in $\sigma$.
Consider the simplex $S:=\conv\{0,v,w_1,w_2\}$.
The length of each edge of $S$ is already determined, either by the parameters alone, or by additionally using the known interior angles via %. This is only non-obvious for the edge $\conv\{w_1,w_2\}$:
\begin{align*}
%&\|v-0\| = r_1,\quad \|w_1-0\| = \|w_2-0\|=r_2, \quad \|v-w_1\|=\|v-w_2\| = \ell, \\
&\|w_1-w_2\|^2= \|w_1-v\|^2+\|w_2-v\|^2 -2\<w_1-v,w_2-v\> \\ &\phantom{\|w_1-w_2\|^2}= 2\ell^2(1-\cos\underbrace{\angle(w_1-v,w_2-v)}_{\alpha_1}).
\end{align*}
Thus, the shape of $S$ is determined.
In particular, this determines the height of the face $\conv\{v,w_1,w_2\}\subset S$ over the vertex $0\in S$.
Since $\aff\{v,w_1,w_2\}=\aff(\sigma)$, this determines the height of $\sigma$ in $P$. This proves $(ii)\implies(iii)$.

\iffalse

Finally, suppose the type $2k$ is known.
If this does not uniquely determine the height of a face, then w.l.o.g., there a is a $2k$-face $\sigma'\in\F_2(P)$ of some height $h'<h$ (where $h$ is the height of $\sigma$).
%
Let $\rho_1<\rho_2$ \resp\ $\rho_1'<\rho_2'$ denote the radii of $\sigma$ \resp\ $\sigma'$.
From $h>h'$ follows via \cref{res:faces_of_bipartite} $(iii)$ that $(*)\;\rho_i<\rho_i'$.
But with \cref{res:angles_between_vertices} we have
%
$$\cos\frac\pi k = \frac{\rho_1^2+\rho_2^2-\ell^2}2 \overset{(*)}< \frac{\rho_1'^2+\rho_2'^2-\ell^2}2 = \cos \frac\pi k.$$
%
This is a contradiction, and we have shown $(i)\implies(iii)$.

\fi

Finally, suppose that the type $2k$ is known.
We then want to show that the~height $h$ is uniquely determined.%
\footnote{The reader motivated to prove this himself should know the following: it is indeed possible to write down a polynomial in $h$ of degree four whose coefficients involve only $r_1, r_2$, $\ell$ and $\cos(\pi/k)$, and whose zeroes include all possible heights of any $2k$-face of $P$.
However, it turns out to be quite tricky to work out which zeroes correspond to feasible solutions.
For certain values of the coefficients there are multiple positive solutions for $h$, some of which correspond to \emph{non-convex} $2k$-faces.
There seems to be no easy way to tell them apart.}
For the sake of contradiction, suppose that the type $2k$ does \emph{not} uniquely determine the height of the face.
Then there is another $2k$-face $\sigma'\in\F_2(P)$ of some height $h'\not=h$.
W.l.o.g.\ assume $h>h'$.

Visualize both faces embedded in $\RR^2$, on top of each other and centered at the origin as shown in the figure below:
\begin{center}
\includegraphics[width=0.33\textwidth]{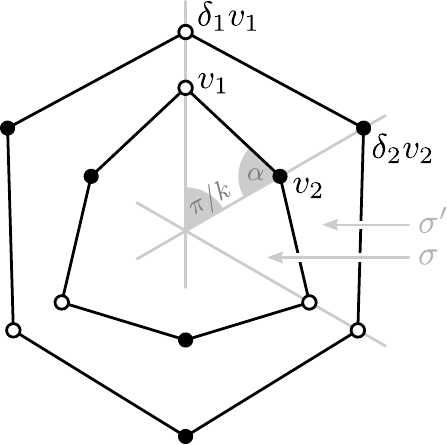}
\end{center}
The vertices in both polygons are equally spaced by an angle of $\pi/k$ (\cf\ \cref{res:angles_between_vertices}) and we can therefore assume that the vertex $v_i$ of $\sigma$ (\resp\ $v'_i$ of $\sigma'$) is a positive multiple of $(\sin(i\pi/k),\cos(i\pi/k))\in\RR^2$ for $i\in\{1,...,2k\}$.
There are then factors $\delta_i\in\RR_+$ with $v_i'=\delta v_i$.
%is a positive multiple of the vertex $v'_i$ of $\sigma'$, that is $v_i=\delta_i v_i'$ for some factor $\delta_i\in\RR_+$.% Also, $\delta_i$ only depends on the partition class $v_i$ and it suffices to talk about $v_1$

The norms of vectors $v_1$, $v_2$, $\delta_1 v_1$ and $\delta_2 v_2$ are the radii of the bipartite polygons $\sigma$ and $\sigma'$.
With \cref{res:faces_of_bipartite} $(iii)$ from $h>h'$ follows $\|v_1\| < \|\delta_1 v_1\|$ and~$\|v_2\|<\|\delta_2 v_2\|$, and thus, $(*)\;\delta_1,\delta_2>1$. W.l.o.g. assume $\delta_1\le\delta_2$.
%W.l.o.g.\ assume $\delta_1\le \delta_2$.

Since both faces have edge length $\ell$, we have $\|v_1-v_2\| = \|\delta_1 v_1-\delta_2 v_2\| = \ell$.
Our goal is to derive the following contradiction:
$$\ell=\|v_1-v_2\| \overset{\mathclap{\smash{(*)}}}< \delta_1 \|v_1-v_2\| = \|\delta_1 v_1-\delta_1 v_2\| \,\overset{\mathclap{\smash{(**)}}}<\, \|\delta_1 v_1-\delta_2 v_2\|=\ell,$$
To prove $(**)$, consider the triangle $\Delta$ with vertices $\delta_1 v_1$, $\delta_2 v_2$ and $\delta_1 v_2$:
\begin{center}
\includegraphics[width=0.29\textwidth]{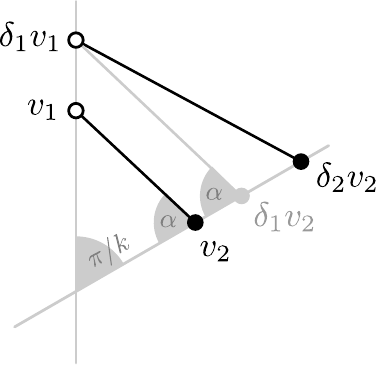}
\end{center}
Since $\sigma$ is convex, the angle $\alpha$ is smaller than $90^\circ$. It follows that the interior angle of $\Delta$ at $\delta_1 v_2$ is obtuse (here we are using $\delta_1\le \delta_2$). Hence, by the sine law, the edge of $\Delta$ opposite to $\delta_1 v_2$ is the longest, which translates to $(**)$.

\iffalse
It remains to prove inequality $(**)$.
This inequality is trivially satisfied if $\delta_1=\delta_2$ and so we can assume $\delta_1<\delta_2$.
We now provide a chain of equivalence transformations of $(**)$ (note the use of $\delta_1-\delta_2<0$ in the third step to reverse the inequality):
%
\begin{align*}
\|\delta_1 v_1-\delta_1 v_2\|^2 &\le \|\delta_1 v_1-\delta_2 v_2\|^2 \\
\delta_1^2 \|v_2\|^2 - 2\delta_1^2\<v_1,v_2\> &\le \delta_2^2 \|v_2\|^2 -2\delta_1\delta_2\<v_1,v_2\> \\
(\delta_1^2 - \delta_2^2) \|v_2\|^2 &\le 2\delta_1(\delta_1-\delta_2)\<v_1,v_2\> \\
(\delta_1 + \delta_2) \|v_2\|^2 &\ge 2\delta_1\<v_1,v_2\> \\
\bar\delta \|v_2\|^2 &\ge \delta_1\<v_1,v_2\>,
\end{align*}
%
where $\bar\delta:=(\delta_1+\delta_2)/2$.
Since $\bar\delta >\delta_1$, it suffices to check $\|v_2\|^2\ge \<v_1,v_2\>$ in order to conclusively prove $(**)$.

Note that $\|v_2\|^2\ge \<v_1,v_2\>$ is equivalent to $\<v_2,v_2-v_1\>\ge 0$, which is equivalent to the statement that the angle $\alpha$ (see figure above) is at most $90^\circ$.
This is true since $\sigma$ is convex (the interior angle is $2\alpha\le 180^\circ$).

We therefore found a contradiction to the assumption that there are two non-congruent $2k$-faces and this proves $(i)\implies(ii),(iii)$.
\fi
\end{proof}

As a consequence of \cref{res:unique_shape}, the interior angles of a face of $P$ do only depend on the type of the face (and the parameters), and so we can introduce the notion of \emph{the} interior angle $\alpha_i^k\in(0,\pi)$ of a $2k$-face at an $i$-vertex.
%
%As a direct consequence of \cref{res:unique_shape}, the notion of \emph{the} interior angle $\alpha_i^k\in(0,\pi)$ of a $2k$-face at an $i$-vertex is well-defined.
Furthermore, set $\eps_k:=(\alpha_1^k-\alpha_2^k)/2\pi$.
By \cref{res:angles} we have $\eps_k> 0$ and
$$\alpha_1^k = \Big(1-\frac1k+\eps_k\Big)\pi,\qquad \alpha_2^k=\Big(1-\frac 1k-\eps_k\Big)\pi.$$

\begin{definition}
If $\tau=(2k_1,...,2k_s)$ is the type of a vertex, then define
	$$K(\tau):=\sum_{i=1}^s \frac 1{k_i},\qquad E(\tau):=\sum_{i=1}^s \eps_{k_i}.$$
%
%If $v\in \F_0(P)$ is of type $\tau$, we also write $K(v)$ \resp\ $E(v)$.
\end{definition}

%Both quantities will play a major role in narrowing down the possible vertex types in $P$.
Both quantities are strictly positive.
%They are linked by a basic geometric observation.

\begin{proposition}\label{res:K_E_ineq}
Let $v\in \F_0(P)$ be a vertex of type $\tau=(2k_1,...,2k_s)$.
\begin{myenumerate}
	\item If $v\in V_1$, then $E(\tau)<K(\tau)-1$ and $s=3$.
	\item If $v\in V_2$, then $E(\tau)>s-2-K(\tau)$.
%	
%	\item If $v\in V_1$, then $s=3$ and 
%	\begin{equation}
%	E(v)<K(v)-1.
%	\end{equation}
%	\item If $v\in V_2$, then
%	\begin{equation}
%	E(v)>s-2-K(v).
%	\end{equation}
\end{myenumerate}
\begin{proof}
Let $\sigma_1,...,\sigma_s\in\F_2(P)$ be the faces incident to $v$, so that $\sigma_j$ is a $2k_j$-face.
The interior angle of $\sigma_j$ at $v$ is $\alpha_i^{\smash{k_j}}$, and the sum of these must be smaller than $2\pi$.
In formulas
$$2\pi > \sum_{j=1}^s \alpha_i^{k_{\smash j}} = \sum_{j=1}^s \Big(1-\frac1{k_j}\pm\eps_{k_j} \Big)\pi = (s-K(\tau)\pm E(\tau))\pi,$$
where $\pm$ is the plus sign for $i=1$, and the minus sign for $i=2$.
Rearranging for $E(v)$ yields $(*)\;{\mp E(\tau)} > s - 2 - K(\tau)$.
If $i=2$, this proves $(ii)$.
If $i=1$, note that from the implication $k_j\ge 2\implies K(\tau)\le s/2$ follows
$$s \overset{(*)}< -E(\tau)+K(\tau)+2 \le 0+\frac s2+2 \quad\implies\quad s<4.$$
The minimum degree of a vertex in a polyhedron is at least three, hence $s=3$, and $(*)$ becomes $(i)$.
\end{proof}
\end{proposition}

This allows us to exclude all but a manageable list of types for 1-vertices.
Note that a vertex $v\in V_1$ has a type of some form $(2k_1,2k_2,2k_3)$.

\begin{corollary}\label{res:1_types}
For a 1-vertex $v\in V_1$ of type $\tau$ holds $K(\tau)>1+E(\tau)>1$. One checks that this leaves exactly the options in \cref{tab:1_types}.
\begin{table}[h!]
\centering
\begin{tabular}{l|l|l}
%type & $K(v)$\\
%\hline
%$\overset{\phantom.}(4,4,\phantom04)$ & $3/2$ \\
%$(4,4,\phantom06)$ & $4/3$ \\
%$(4,4,\phantom08)$ & $5/4$ \\
%$(4,4,10)$ & $6/5$ \\
%$(4,4,12)$ & $7/6$ \\[-0.4ex]
%$\vdots\quad\;\;$ & $\;\;\vdots$ \\[0.3ex]
%$(4,4,2k)$ & $1+1/k$ \\[0.5ex]
%\hline
%$\overset{\phantom.}(4,6,\phantom06)$ & $7/6$ \\
%$(4,6,\phantom08)$ & $13/12$ \\
%$(4,6,10)$ & $31/30$
%
$\tau$ & $K(\tau)$ & $\Gamma$ \\
\hline
$\overset{\phantom.}(4,4,\phantom04)$ & $3/2$ & $I_1\oplus I_1\oplus I_1$ \\
$(4,4,\phantom06)$ & $4/3$ & $I_1\oplus I_2(3)$ \\
$(4,4,\phantom08)$ & $5/4$ & $I_1\oplus I_2(4)$ \\
%$(4,4,10)$ & $6/5$ & $I_1\oplus I_2(5)$ \\
%$(4,4,12)$ & $7/6$ & $I_1\oplus I_2(6)$ \\%
[-0.4ex]
$\quad\;\;\vdots\;\;$ & $\;\;\vdots$ & $\quad\;\vdots$ \\[0.3ex]
$(4,4,2k)$ & $1+1/k$ & $I_1\oplus I_2(k)$ \\[0.5ex]
\hline
$\overset{\phantom.}(4,6,\phantom06)$ & $7/6$ & $A_3=D_3$ \\
$(4,6,\phantom08)$ & $13/12$ & $B_3$ \\
$(4,6,10)$ & $31/30$ & $H_3$
\end{tabular}
\caption{Possible types of 1-vertices, their $K$-values and the $\Gamma$ of the $\Gamma$-permutahedron in which all vertices have this type.}
\label{tab:1_types}
\end{table}
\end{corollary}

The types in \cref{tab:1_types} are called the \emph{possible types} of 1-vertices.
Each of the possible types is realizable in the sense that there exists a bipartite polyhedron in which all 1-vertices have this type.
Examples are provided by the $\Gamma$-permutahedra (the $\Gamma$ of that $\Gamma$-permutahedron is listed in the right column of \cref{tab:1_types}).
These are not \emph{strictly} bipartite though.

The convenient thing about $\Gamma$-permutahedra is that all their vertices are of the same type.
We cannot assume this for general strictly bipartite polyhedra, not even for all 1-vertices.

\subsection{Spherical polyhedra}

The purpose of this section is to define a second notion of interior angle for each face.
These angles can be defined in several equivalent~ways, one of which is via spherical polyhedra.

A \emph{spherical polyhedron} is an embedding of a planar graph into the unit sphere, so that all edges are embedded as great circle arcs, and all regions are convex\footnote{Convexity on the sphere means that the shortest great circle arc connecting any two points in the region is also contained in the region.}.
If $0\in\mathrm{int}(P)$, we can associate to $P$ a spherical polyhedron $P^S$ by applying central projection
$$\RR^3\nozero\to\Sph_1(0),\quad x\mapsto\frac x{\|x\|}$$
to all its vertices and edges (this process is visualized below).
\begin{center}
\includegraphics[width=0.42\textwidth]{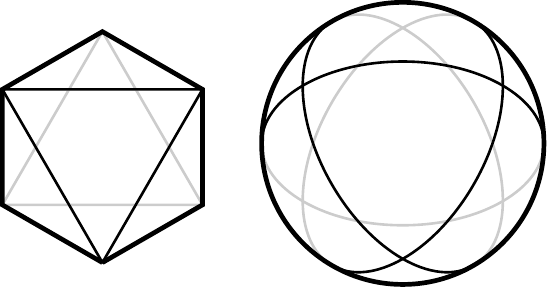}
\end{center}

The vertices, edges and faces of $P$ have spherical counterparts in $P^S$ obtained as projections onto the unit sphere.
Those will be denoted with a superscript \enquote{$S$\,}.
For example, if $e\in\F_1(P)$ is an edge of $P$, then $e^{\smash S}$ denotes the corresponding \enquote{spherical edge}, which is a great circle arc obtained as the projection of $e$ onto the sphere.

%Considering the spherical polyhedron of the bipartite polyhedron $P$ will be an important tool for our further investigation.
%To each vertex $v$, edge $e$ or face $\sigma$ of $P$ we can associate its projection $v^S$, $e^S$ or $\sigma^S$ as vertex, edge or faces of $P^S$. 
%The \emph{spherical polygon} $\sigma^S$ also has interior angles at each incident vertex, and we will establish, that again, its value only depends on the type of the involved face and the partition class of the involved vertex.
%
We still need to justify that the spherical polyhedron of $P$ is well-defined, by proving that $P$ contains the origin:

\begin{proposition}\label{res:0_in_P}
$0\in\mathrm{int}(P)$.
\begin{proof}
The proof proceeds in several steps.

\textit{Step 1}:
Fix a 1-vertex $v\in V_1$ with neighbors $w_1,w_2,w_3\in V_2$, and let $u_i:=w_i-v$ be the direction of the edge $\conv\{v,w_i\}$ emanating from $v$.
Let $\sigma_{ij}\in\F_2(P)$ denote the $2k$-face containing $v,w_i$ and $w_j$.
The interior angle of $\sigma_{ij}$ at $v$ is then $\angle(u_i,u_j)$, which by \cref{res:angles} and $k\ge 2$ satisfies %estimates as
$$\angle(u_i,u_j)>\Big(1-\frac1k\Big)\pi \ge \frac\pi2\quad\implies\quad \<u_i,u_j\><0.$$

\textit{Step 2}:
Besides $v$, the polyhedron $P$ contains another 1-vertex $v'\in V_1$.
It then holds $v'\in v+\cone\{u_1,u_2,u_3\}$, which  means that there are non-negative coefficients $a_1,a_2,a_3\ge 0$, at least one positive, so that $v+a_1 u_1+ a_2u_2 + a_3u_3 = v'$.
Rearranging and applying $\<v,\free\>$ yields
\begin{align*}
a_1\<v,u_1\>+a_2\<v,u_2\>+a_3\<v,u_3\>
&=\<v,v'\>-\<v,v\> 
\tag{$*$}
\\&= r_1^2\cos\angle(v,v')-r_1^2 < 0.
\tag*{}
\end{align*}
The value $\<v,u_i\>$ is independent of $i$ (see \cref{res:angles_between_vertices}).
Since there is at least one positive coefficient $a_i$, from $(*)$ follows $\<v,u_i\><0$.\footnote{Note that this provides the formal proof mentioned in \cref{rem:alternative_definiton}, namely, that the triangle $\conv\{0,v_1,v_2\}$ is acute at $v_1$ and $v_2$.}

\textit{Step 3}:
By the previous steps, $\{v,u_1,u_2,u_3\}$ is a set of four vectors with pair-wise negative inner product.
The convex hull of such an arrangement in 3-dimensional Euclidean space does necessarily contain the origin in its interior, or equivalently, there are positive coefficients $a_0,...,a_3>0$ with $a_0 v+a_1 u_1+a_2 u_2+a_3 u_3=0$ (for a proof, see \cref{res:sum_to_zero}).
In other words: $0\in v+\mathrm{int}(\cone\{u_1,u_2,u_3\})$.

\textit{Step 4}:
If $H(\sigma)$ denotes the half-space associated with the face $\sigma\in\F_2(P)$, then
$$0\in v+\mathrm{int}(\cone\{u_1,u_2,u_3\}) = \bigcap_{\sigma\sim v} \mathrm{int}(H(\sigma)).$$
Thus, $0\in\mathrm{int}(H(\sigma))$ for all faces $\sigma$ incident to $v$.
But since every face is incident to a 1-vertex, we obtain $0\in\mathrm{int}(H(\sigma))$ for all $\sigma\in\F_2(P)$, and thus $0\in\mathrm{int}(P)$ as well.

\end{proof}
\end{proposition}

The main reason for introducing spherical polyhedra is that we can talk about the \emph{spherical interior angles} of their faces.

Let $\sigma\in \F_2(P)$ be a face, and $v\in\F_0(\sigma)$ one of its vertices.
Let $\alpha(\sigma,v)$ denote the interior angle of $\sigma$ at $v$, and $\beta(\sigma, v)$ the spherical interior angle of $\sigma^S$ at $v^S$.
It only needs a straight-forward computation (involving some spherical geometry) to establish a direct relation between these angles: \eg\ if $v$ is a 1-vertex, then
$$\sin^2(\ell^S)\cdot (1-\cos\beta(\sigma,v)) = \Big(\frac{\ell}{r_2}\Big)^2\!\!\cdot (1-\cos\alpha(\sigma,v)),$$
where $\ell^S$ denotes the arc-length of an edge of $P^S$ (indeed, all edges are of the same length).
An equivalent formula exists for 2-vertices.
The details of the computation are not of relevance, but can be found in \cref{sec:computations}.

The core message is that the value of $\alpha(\sigma,v)$ uniquely determines the value of $\beta(\sigma,v)$ and vice versa.
In particular, since the value of $\alpha(\sigma,v)=\alpha_i^{\smash k}$ does only depend on the type of the face and the partition class of the vertex, so does $\beta(\sigma,v)$, and it makes sense to introduce the notion $\beta_i^k$ for the spherical interior angle of a $2k$-gonal spherical face of $P^S$ at (the projection of) an $i$-vertex.
Thus, we have
\begin{equation}\label{eq:angle_relations}
\beta_i^{k_1}=\beta_i^{k_2}\quad\Longleftrightarrow\quad\alpha_i^{k_1}=\alpha_i^{k_2} \quad\overset{\ref{res:unique_shape}}\Longleftrightarrow\quad k_1=k_2,
\end{equation}
where we use \cref{res:unique_shape} for the last equivalence.

\begin{observation}\label{res:spherical_interior_angles_are_good}
The spherical interior angles $\beta_i^k$ have the following properties:
\begin{myenumerate}
	\item The spherical interior angles surrounding a vertex add up to exactly $2\pi$. That is, for an $i$-vertex $v\in\F_0(P)$ of type $(2k_1,...,2k_s)$ holds
	$$\beta_i^{k_1}+\cdots+\beta_i^{k_s} = 2\pi.$$
	
	\item
	The sum of interior angles of a spherical polygon always exceed the interior angle sum of a respective flat polygon. That is, it holds
	$$k \beta_1^k+k\beta_2^k > 2(k-1)\pi\quad\Longrightarrow\quad \beta_1^k+\beta_2^k > 2\Big(1-\frac1k\Big)\pi.$$
\end{myenumerate}
\end{observation}

\iffalse

\begin{observation}\label{res:spherical_geometry_facts}
We note two relevant facts:
%
\begin{myenumerate}
	\item 
	The (arc-)length $\ell^S$ of a (spherical) edge of $P^S$ is exactly the angle between its end vertices.
	%More precisely, if $e=\conv\{v,w\}$ is an edge with end vertices $v,w\in\F_0(P)$, then the lenght of $e^S$ is the angle $\angle(v,w)$ between the vectors $v$ and $w$.
	We determined in \cref{res:angles_between_vertices}
	%
	$$\ell^S = \arccos\Big(\frac{r_1^2+r_2^2-\ell^2}{2r_1 r_2}\Big),$$
	%
	in particular, $P^S$ has all edges of the same length.
		
	\item
	A face $\sigma^S$ of $P^S$ is a convex spherical polygon with well-defined \emph{(spherical) interior angles}.
	The sum of interior angles of a spherical polygon is larger than the one of a respective flat polygon, and depends on its area.
	However, the sum of the interior angles surrounding a vertex is exactly $2\pi$.
	This is the central property we apply later in many circumstances.
	
%	\item The spherical counterpart to the classical law of cosine is the \emph{spherical law of cosine}.
%	If we have a spherical triangle with edges of length $a$, $b$ and $c$, and an interior angle $\gamma$ opposite to $c$, then
%	%
%	$$\cos(c)=\cos(a)\cos(b)+\sin(a)\sin(b)\cos(\gamma).$$
\end{myenumerate}
\end{observation}

\fi

This has some consequences for the strictly bipartite polyhedron $P$:

\begin{corollary}\label{res:incompatible_1_types}
$P$ contains at most two different types of 1-vertices, and if there are two, then one is of the form $(4,4,2k)$, and the other one is of the form $(4,6,2k')$ for distinct $k\not=k'$ and $2k'\in\{6,8,10\}$.
%$P$ contains at most two different types of 1-vertices.
%In detail, there are distinct fixed $k,k'\in\NN,2k'\in\{6,8,10\}$ so that every 1-vertex of $P$ is either of type $(4,4,2k)$ or of type $(4,6,2k')$.
%
\begin{proof}
Each possible type listed in \cref{tab:1_types} is either of the form $(4,4,2k)$ or of the form $(4,6,2k')$ for some $2k\ge 4$ or $2k'\in\{6,8,10\}$.

%Two types of the form $(4,4,2k)$ are metrically incompatible as they only differ in the last entry.
%They cannot occur together in in $P$ by \cref{res:incompatible_types}.
%The same holds for two types of the form $(4,6,2k')$.
%
%For the same reason we need $k\not= k'$, as $(4,4,2k)$ and $(4,6,2k)$ differ only in the second place.
%The type $(4,6,2k')$ appears in \cref{tab:1_types} only for $2k'\in\{4,6,8,10\}$.
%But setting $2k'=4$ gives the type $(4,6,4)$, which is the same as $(4,4,6)$ and thus metrically incompatible with $(4,4,2k),2k\not=6$.

%We now apply \cref{res:spherical_interior_angles} $(i)$ and $(ii)$ to deduce information about simultaneously occurring types.
If $P$ contains simultaneously 1-vertices of type $(4,4,2k_1)$ and $(4,4,2k_2)$, apply \cref{res:spherical_interior_angles_are_good} $(i)$ to see
$$
\beta_1^2+\beta_1^2+\beta_1^{k_1}
\overset{(i)}=\beta_1^2+\beta_1^2+\beta_1^{k_2} 
\; \implies \; 
\beta_1^{k_1}=\beta_1^{k_2}
\; \overset{\eqref{eq:angle_relations}}\implies \; 
k_1=k_2.$$
If $P$ contains simultaneously 1-vertices of type $(4,6,2k_1')$ and $(4,6,2k_2')$, then
$$\beta_1^2+\beta_1^3+\beta_1^{k_1'}
\overset{(i)}=\beta_1^2+\beta_1^3+\beta_1^{k_2'} 
\; \implies \;
\beta_1^{k_1'}=\beta_1^{k_2'}
\; \overset{\eqref{eq:angle_relations}}\implies \;
 k_1'=k_2'.$$
Finally, if $P$ contains simultaneously 1-vertices of type $(4,4,2k)$ and $(4,6,2k')$, then
$$\beta_1^2+\beta_1^2+\beta_1^{k}\overset{(i)}=\beta_1^2+\beta_1^3+\beta_1^{k'} \; \implies \; \beta_1^k-\beta_1^{k'} = \underbrace{\beta_1^3-\beta_1^2}_{\text{$\not=0$ by \eqref{eq:angle_relations}}} \;\overset{\eqref{eq:angle_relations}}\implies \; k\not=k'.$$
\end{proof}
\end{corollary}

%\begin{definition}\label{def:kk_permutahedron}
%A \emph{$(k,k')$-polyhedron} as a bipartite polyhedron in which all 1-vertices are of type $(4,4,2k)$ or $(4,6,k')$.
%\end{definition}

%\Cref{res:incompatible_1_types} then states that $P$ is a $(k,k')$-polyhedron for some distinct $k\not=k'$ and $k'\in\{6,8,10\}$.
%Since each edge of $P$ contains a 1-vertex, we observe the following:

Since each edge of $P$ is incident to a 1-vertex, we obtain

\begin{observation}\label{res:edge_types}
If $P$ has only 1-vertices of types $(4,4,2k)$ and $(4,6,2k')$,  then each edge of $P$ is of one of the types
$$\underbrace{(4,4), \; (4,2k)}_{\mathclap{\text{\textup{from a $(4,4,2k)$-vertex}}}}, \; \underbrace{(4,6), \; (4,2k') \quad\text{or}\quad (6,2k')}_{\text{\textup{from a $(4,6,2k')$-vertex}}}.$$
\end{observation}

%For the next application, r

%Recall that the \emph{dihedral angle} of an edge in $P$ is the angle between its two incident faces measured on the inside of the polyhedron.

%Recall further, that for a \emph{simple} vertex (that is, a vertex of degree three), the interior angles of the incident faces already uniquely determine the dihedral angles of the incident edges, and vice versa (see Appendix, \cref{res:simple_vertex}).

\begin{corollary}\label{res:dihedral_angles}
The dihedral angle of an edge $e\in\F_1(P)$ of $P$ only depends on its type.
%\msays{Do we need this result at some point other than \cref{res:1_2_vertex_same_type}? Otherwise, delete and integrate in the proof of \cref{res:1_2_vertex_same_type}.}
%
\begin{proof}
Suppose that $e$ is a $(2k_1,2k_2)$-edge. 
Then $e$ is incident to a 1-vertex $v\in V_1$ of type $(2k_1,2k_2, 2k_3)$.
By \cref{res:spherical_interior_angles_are_good} $(i)$ holds $\beta_1^{\smash{k_3}}=2\pi-\beta_1^{\smash{k_1}}-\beta_1^{\smash{k_2}}$, which further determines $k_3$.
By \cref{res:unique_shape} we have uniquely determined interior angles $\alpha_1^{\smash{k_1}},\alpha_1^{\smash{k_2}}$ and $\alpha_1^{\smash{k_3}}$.

It is known that for a simple vertex (that is, a vertex of degree three) the interior angles of the incident faces already determine the dihedral angles at the incident edges (for a proof, see the Appendix, \cref{res:simple_vertex}).
Consequently, the dihedral angle at $e$ is already determined.
\end{proof}
\end{corollary}

The next result shows that $\Gamma$-permutahedra are the only bipartite polytopes in which a 1-vertex and a 2-vertex can have the same type.

\begin{corollary}\label{res:1_2_vertex_same_type}
$P$ cannot contain a 1-vertex and a 2-vertex of the same type.
\begin{proof}
Let $v\in\F_0(P)$ be a vertex of type $(2k_1,2k_2,2k_3)$.
The incident edges are of type $(2k_1,2k_2)$, $(2k_2,2k_3)$ and $(2k_3,2k_1)$ respectively.
By \cref{res:dihedral_angles} the dihedral angles of these edges are uniquely determined, and since $v$ is simple (that is, has degree three), the interior angles of the incident faces are also uniquely determined (\cf\ Appendix, \cref{res:simple_vertex}).
In particular, we obtain the same angles independent of whether $v$ is a 1-vertex or a 2-vertex.

A 1-vertex is always simple, and thus, a 1-vertex and a 2-vertex of the same type would have the same interior angles at all incident faces, that is, $\alpha_1^k=\alpha_2^k$ for each incident $2k$-face.
But this is not possible if $P$ is \emph{strictly} bipartite (by \cref{res:faces_of_bipartite} $(ii)$ and \cref{res:angles}).
\end{proof}
\end{corollary}

%We have seen that much information can be extracted quite easily using the spherical interior angles.
%However, to proceed further, we have to
%We now proceed extracting further information via
% the (classical) interior angles of the faces of $P$.

%We now return to investigating the \enquote{classical} interior angles of the faces of $P$.

\iffalse

\newpage
\hrulefill

The \emph{dihedral angle} of an edge $e\in\F_1(P)$ is the angle between its incident faces.
In a bipartite polyhedron, this angle is determined by the type of $e$.

\begin{proposition}\label{res:dihedral_angle}
The dihedral angle of an $e\in\F_1(P)$ is uniquely determined by the type of the edge.
%
\begin{proof}
Suppose $e$ is of type $(2k_1,2k_2)$.
Let $\sigma_1,\sigma_2\in\F_2(P)$ be the faces incident to $e$, so that $\sigma_i$ is of type $2k_i$.
By \cref{res:unique_shape}, the type $2k_i$ uniquely determines the height $h_i$ of $\sigma_i$.
There are exactly two planes through $E$ of distance $h_i$ from the origin.

\TODO \msays{: Why does $P$ contain the origin?}

The affine span $\aff(\sigma_i)$ is the uniquely determined plane of distance $h_i$ from the origin that contains the line $\aff(e)$.
The dihedral angle at $e$ is the angle between these planes.
\end{proof}
\end{proposition}

\fi

\subsection{Adjacent pairs}

%We now proceed excluding further types from occuring as types of 1-vertices.

Given a 1-vertex $v\in V_1$ of type $\tau_1=(2k_1,2k_2,2k_3)$, for any two distinct $i,j\in\{1,2,3\}$, $v$ has a neighbor $w\in V_2$ of type $\tau_2=(2k_i,2k_j,*,...,*)$, where $*$ are placeholders for unknown entries.
The pair of types
$$(\tau_1,\tau_2)=((2k_1,2k_2,2k_3),\,(2k_i,2k_j,*,...,*))$$
%
%
%Suppose that $v\in V_1$ is a 1-vertex of type $\tau_1=(4,2k_1,2k_2)$.
%Then $v$ is adjacent to a 2-vertex $w\in V_2$ of type $\tau_2=(2k_1,2k_2,*,...,*)$, where $*$ are placeholders for unknown (and unimportant) entries.
%The pair of types
%%
%$$(\tau_1,\tau_2)=((4,2k_1,2k_2),\,(2k_1,2k_2,*,...,*))$$
%%
is called an \emph{adjacent pair} of $P$.
It is the purpose of this section to show that certain adjacent pairs cannot occur in $P$.
Excluding enough adjacent pairs for fixed $\tau_1$ then proves that the type $\tau_1$ cannot occur as the type of a 1-vertex.

Our main tools for achieving this will be the inequalities established in \cref{res:K_E_ineq} $(i)$ and $(ii)$, that is
$$E(\tau_1)\overset{\mathclap{(i)}}<K(\tau_1)-1\quad\text{and}\quad E(\tau_2)\overset{\mathclap{(ii)}}>s-2-K(\tau_2),$$
where $s$ is the number of elements in $\tau_2$.
For a warmup, and as a template for~fur\-ther calculations, we prove that the adjacent pair $(\tau_1,\tau_2)=((4,6,8),(6,8,8))$ will not occur in $P$.

\begin{example}\label{ex:infeasible}
By \cref{res:K_E_ineq} $(i)$ we have
$$(*)\quad \eps_2+\eps_3+\eps_4=E(\tau_1)\overset{(i)}<K(\tau_1)-1 = \frac12+\frac13+\frac14-1=\frac1{12}.$$
On the other hand, by \cref{res:K_E_ineq} $(ii)$ we have
\begin{align*}
(**)\quad \frac2{12} = 3-2-\Big(\frac13+\frac14+\frac14\Big)&=s-2-K(\tau_2) \\[-1ex]
&\overset{(ii)}< E(\tau_2) = \underbrace{\eps_3+\eps_4}_{< 1/12}+\underbrace{\eps_4}_{\mathclap{<1/12}} < \frac2{12},
\end{align*}
which is a contradiction.
Hence this adjacent pair cannot occur.
Note that we used $(*)$ to upperbound certain sums of $\eps_i$ in $(**)$.
\end{example}

An adjacent pair excluded by using the inequalities from \cref{res:K_E_ineq} $(i)$ and $(ii)$ as demonstrated in \cref{ex:infeasible} will be called \emph{infeasible}.

The argument applied in \cref{ex:infeasible} will be repeated many times for many different adjacent pairs in the upcoming sections \cref{sec:468,sec:4610,sec:466,sec:442k}, and we shall therefore use a tabular form to abbreviate it.
After fixing, $\tau_1=(4,6,8)$, the argument to refute the adjacent pair $(\tau_1,\tau_2)=((4,6,8),(6,8,8))$ is abbreviated in the first row of the following table:
\begin{center}
\begin{tabular}{l|rcll}
$\tau_2$ & $s-2-K(\tau_2)$ & $\overset?<$ & $E(\tau_2)$ & \\[0.4ex]
\hline
$\overset{\phantom.}(6,8,8)$ & $2/12$ & $\not<$ & $(\eps_3+\eps_4)+\eps_4$ & $<$
 $2/12$ \\
$(6,8,6,6)$ & $9/12$ & $\not<$ & $(\eps_3+\eps_4)+\eps_3+\eps_3$ & $<$ $3/12$
\end{tabular}
\end{center}
The second row displays the analogue argument for another example, namely, the pair $((4,6,8),(6,8,6,6))$, showing that it is infeasible as well.
Both rows will reappear in the table of \cref{sec:468} where we exclude $(4,6,8)$ as a type for 1-vertices entirely.
Note that the terms in the column below $E(\tau_2)$ are grouped by parenthesis to indicate which subsums are upper bounded via \cref{res:K_E_ineq} $(i)$.
In this example, if there are $n$ groups, then the sum is upper bounded by $n/12$.

The placeholders in an adjacent pair $((2k_1,2k_2,2k_3),(2k_i,2k_j,*,...,*))$ can, in theory, be replaced by an arbitrary sequence of even numbers, and each such pair has to be refuted separately.
The following fact will make this task tractable: write $\tau\subset \tau'$ if $\tau$ is a \emph{subtype} of $\tau'$, that is, a vertex type that can be obtained from $\tau'$ by removing some of its entries.
We then can prove

\begin{proposition}\label{res:subtype}
If $(\tau_1,\tau_2)$ is an infeasible adjacent pair, then the pair $(\tau_1,\tau_2')$ is infeasible as well, for every $\tau_2'\supset \tau_2$.
\begin{proof}
Suppose $\tau_2=(2k_1,...,2k_s)$, $\tau_2'=(2k_1,...,2k_s,2k_{s+1},...,2k_{s'})\supset\tau_2$, and that the pair $(\tau_1,\tau_2')$ is \emph{not} infeasible.
Then $\tau_2'$ satisfies \cref{res:K_E_ineq} $(ii)$
%
%\begin{align*}
%\eps_{k_1}+\cdots+\eps_{k_{s'}} &> s'-2-\Big(\frac1{k_1}+\cdots+\frac1{k_{s'}}\Big) \\[-3.2ex]
%\implies\quad \eps_{k_1}+\cdots+\eps_{k_s} &> s - 2 - \Big(\frac1{k_1}+\cdots+\frac1{k_{s}}\Big) + \sum_{i=s+1}^{s'}\!\!\overbrace{\Big(1-\frac1{k_i}-\eps_{k_i}\Big)}^{\alpha_2^{k_i}/\pi > 0} \\
%&> s - 2 - \Big(\frac1{k_1}+\cdots+\frac1{k_{s}}\Big).
%\end{align*}
\begin{align*}
E(\tau_2') &> s'-2-K(\tau_2') \\[-5ex]
\implies\quad E(\tau_2) &> s - 2 - K(\tau_2) + \sum_{i=s+1}^{s'}\!\!\overbrace{\Big(1-\frac1{k_i}-\eps_{k_i}\Big)}^{\alpha_2^{k_i}/\pi > 0} > s - 2 - K(\tau_2).
\end{align*}
But this is exactly the statement that $\tau_2$ satisfies \cref{res:K_E_ineq} $(ii)$ as well, \ie\ that the pair $(\tau_1,\tau_2)$ is also not infeasible.
\end{proof}
\end{proposition}

By \cref{res:subtype} it is sufficient to exclude so-called \emph{minimal infeasible adjacent pairs}, that is, infeasible adjacent pairs $(\tau_1,\tau_2)$ for which $(\tau_1,\tau_2')$ is not infeasible for any $\tau_2'\subset\tau_2$.

A second potential problem is, that we know little about the values that might replace the placeholders in $\tau_2=(2k_i,2k_j,*,...,*)$.
For our immediate goal, dealing with the following special case is sufficient:

\begin{proposition}\label{res:exclusive_types}
The placeholders in an adjacent pair $((4,6,2k'),(6,2k',*,...,*))$ can only contain $4$, $6$ and $2k'$.
\begin{proof}
Suppose that $P$ contains an adjacent pair $$(\tau_1,\tau_2)=((4,6,2k'),(6,2k',2k,*,...,*))$$ induced by a 1-vertex $v\in V_1$ of type $\tau_1$ with neighbor $w\in V_2$ of type $\tau_2$.
Suppose further that $2k\not\in\{4,6,2k'\}$.
%There must then be a $2k$-face incident to a 1-vertex $u$ of a type other than $(4,6,2k')$.
The vertex $w$ is then incident to a $2k$-face, and therefore also to a 1-vertex $u\in V_1$ of type $(4,4,2k)$ ($u$ cannot be of type $(4,6,2k)$ because of $k\not=k'$ and \cref{res:incompatible_1_types}).
%By \cref{res:incompatible_1_types}, $u$ must be of type $(4,4,2k)$.
This configuration is depicted below:
\begin{center}
\includegraphics[width=0.34\textwidth]{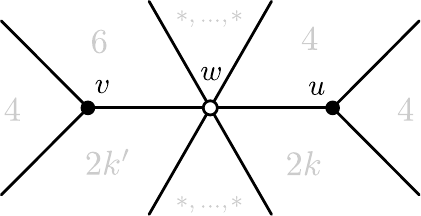}
\end{center}
Note that $w$ is also incident to a $4$-face, and thus $(6,2k',2k,4)\subseteq \tau_2$.

By \cref{res:K_E_ineq} $(i)$ the existence of 1-vertices of type $(4,4,2k)$ and $(4,6,2k')$ yields inequalities
\begin{equation}
\label{eq:two_inequalities}
\eps_2+\eps_2+\eps_k<\frac1k \quad\text{and}\quad\eps_2+\eps_3+\eps_{k'}<\frac1{k'}-\frac16.
\end{equation}
%
%As noted in \cref{res:edge_types}, the only edges incident to $2k$-faces are of type $(4,2k)$.
%This means, $w$ must be incident to a 4-face, and the placeholders in $\tau_2$ must also contain a 4.
Since $\tau_2$ has $\tau:=(6,2k',2k,4)$ as a subtype, by \cref{res:subtype} it suffices to show that the pair $((4,6,2k'),(6,2k',2k,4))$ is infeasible.
This follows via \cref{res:K_E_ineq} $(ii)$:
\begin{align*}
\frac76-\frac1k-\frac1{k'}
%=4-2-\Big(\frac13+\frac1{k'}+\frac1k+\frac12\Big) 
=\;&4-2-K(\tau)\\\overset{(ii)}< \;&E(\tau) = \underbrace{\eps_2+\eps_3+\eps_{k'}}_{<1/k'-1/6}+\underbrace{\eps_k}_{\mathclap{<1/k}} \!\!\overset{\eqref{eq:two_inequalities}}< \frac1k+\frac1{k'}-\frac16,
\end{align*}
which rearranges to $1/k+1/{k'}>2/3$.
Recalling  $2k'\in\{6,8,10\}\implies k'\ge 3$ (from \cref{res:incompatible_1_types}) and $2k\not\in\{4,6,2k'\}\implies k\ge 4$ shows that this is not possible.
\end{proof}
\end{proposition}

%Having established these methods, we proceed by excluding whole types of 1-vertices in $P$.

\subsection{The case $\tau_1=(4,6,10)$}
\label{sec:4610}

If $P$ contains a 1-vertex of type $(4,6,10)$, then it contains an adjacent pair of the form
$$(\tau_1,\tau_2)=((4,6,10),(6,10,*,...,*)).$$
We proceed as demonstrated in \cref{ex:infeasible}.
\cref{res:K_E_ineq} $(i)$ yields $\eps_2+\eps_3+\eps_5 < 1/30$.
By \cref{res:exclusive_types} the placeholders can only take on values 4, 6 or $10$.
The following table lists the minimally infeasible adjacent pairs and proves their infeasibility.
\begin{center}
\begin{tabular}{l|rcll}
$\tau_2$ & $s-2-K(\tau_2)$ &$\overset?<$&  $E(\tau_2)$ & \\[0.4ex]
\hline
$\overset{\phantom.}(6,10,6)$ & $\phantom04/30$ &$\not<$& $(\eps_3+\eps_5)+\eps_3$ & $< 2/30$ \\
$(6,10,10)$ & $\phantom08/30$ &$\not<$& $(\eps_3+\eps_5)+\eps_5$ & $< 2/30$ \\
$(6,10,4,4)$ & $14/30$ &$\not<$& $(\eps_2+\eps_3+\eps_5)+\eps_2$ & $< 2/30$ \\
%$(6,10,2k)$ & $14/30-1/k$ & $<$ & $(\eps_2+\eps_3)+\eps_k$ & $ < 1/30 + 1/k$
\end{tabular}
\end{center}
%
%The listed types are therefore excluded for $\tau_2$.
By \cref{res:subtype} we conclude: the placeholder in $\tau_2=(6,10,*,...,*)$ can contain no 6 or 10, and at most one 4.
This leaves us with the option $\tau_2=(4,6,10)$, which is the same as $\tau_1$ and therefore not possible by \cref{res:1_2_vertex_same_type}.
Therefore, $P$ cannot contain a 1-vertex of type $(4,6,10)$.

\subsection{The case $\tau_1=(4,6,8)$}
\label{sec:468}

If $P$ contains a 1-vertex of type $(4,6,8)$, then it also contains an adjacent pair of the form
$$(\tau_1,\tau_2)=((4,6,8),(6,8,*,...,*)).$$
We proceed as demonstrated in \cref{ex:infeasible}.
\cref{res:K_E_ineq} $(i)$ yields $\eps_2+\eps_3+\eps_4 < 1/12$.
By \cref{res:exclusive_types} the placeholders can only take on values 4, 6 or 8.
The following table lists the minimally infeasible adjacent pairs and proves their infeasibility.
\begin{center}
\begin{tabular}{l|rcll}
$\tau_2$ & $s-2-K(\tau_2)$ &$\overset?<$&  $E(\tau_2)$ & \\[0.4ex]
\hline
$\overset{\phantom.}(6,8,8)$ & $2/12$ & $\not<$ & $(\eps_3+\eps_4)+\eps_3$ & $< 2/12$ \\
$(6,8,4,4)$ & $5/12$ &$\not<$& $(\eps_2+\eps_3+\eps_4)+\eps_2$ & $< 2/12$ \\
$(6,8,4,6)$ & $7/12$ &$\not<$& $(\eps_2+\eps_3+\eps_4)+\eps_3$ & $< 2/12$ \\
$(6,8,6,6)$ & $9/12$ &$\not<$& $(\eps_2+\eps_3+\eps_4)+\eps_3 + \eps_3$ & $< 3/12$ 
\end{tabular}
\end{center}
%
%The listed types are therefore excluded for $\tau_2$.
By \cref{res:subtype} we conclude: the placeholder in $\tau_2=(6,8,*,...,*)$ can contain no 8, and at most one 4 or 6, but not both at the same time.

This leaves us with the options $\tau_2=(4,6,8)$ and $\tau_2=(6,6,8)$.
In the first case, $\tau_1=\tau_2$ which not possible by \cref{res:1_2_vertex_same_type}.
In the second case, there would be two adjacent 6-faces, but $P$ does not contain $(6,6)$-edges by \cref{res:edge_types} with $2k'=8$.
Therefore, $P$~cannot contain a 1-vertex of type $(4,6,8)$.

\subsection{The case $\tau_1=(4,6,6)$}
\label{sec:466}

If $P$ contains a 1-vertex of type $(4,6,6)$, then it also contains an adjacent pair of the form
$$(\tau_1,\tau_2)=((4,6,6),(6,6,*,...,*)).$$
We proceed as demonstrated in \cref{ex:infeasible}.
\cref{res:K_E_ineq} $(i)$ yields $\eps_2+\eps_3+\eps_3 < 1/6$.
By \cref{res:exclusive_types} the placeholders can only take on values 4 or 6.
The following table lists the minimally infeasible adjacent pairs and proves their infeasibility.
\begin{center}
\begin{tabular}{l|rcll}
$\tau_2$ & $s-2-K(\tau_2)$ &$\overset?<$&  $E(\tau_2)$ & \\[0.4ex]
\hline
%$\overset{\phantom.}(6,6,6,2k)$ & $1-1/k = (k-1)/k$ &<& $(\eps_3+\eps_3)+\eps_3+\eps_k$ & < $2/6+1/k$ \\\\
$\overset{\phantom.}(6,6,4,4)$ & $2/6$ & $\not<$ & $(\eps_2+\eps_3+\eps_3)+\eps_2$ & $<$ $2/6$ \\
$(6,6,6,4)$ & $3/6$ & $\not<$ & $(\eps_2+\eps_3+\eps_3)+\eps_3$ & $<$ $2/6$ \\
$(6,6,6,6)$ & $4/6$ & $\not<$ & $(\eps_3+\eps_3)+(\eps_3+\eps_3)$ & $<$ $2/6$
\end{tabular}
\end{center}
%
%The listed types are therefore excluded for $\tau_2$.
By \cref{res:subtype} we conclude: the placeholder in $\tau_2=(6,6,*,...,*)$ can contain at most one 4 or 6, but not both at the same time.

This leaves us with the options $\tau_2=(4,6,6)$ and $\tau_2=(6,6,6)$.
In the first case we have $\tau_1=\tau_2$, which is not possible by \cref{res:1_2_vertex_same_type}.
Excluding $(6,6,6)$ needs more work:
%
%suppose that $v$ is a 2-vertex of type $(6,6,6)$. Fix an incident 6-gon $\sigma\in\F_2(P)$. The neighbors of $v$ in $\sigma$ must then be 1-vertices of type $(4,6,6)$ (since by \cref{res:incompatible_1_types} any potential other type of 1-vertex does not contain a six). Consider \cref{fig:...} to see that the 2-vertices at distance to $v$ in $\sigma$ mu
%
%\iffalse
fix a 6-gon $\sigma\in\F_2(P)$.
Each edge of $\sigma$ is either of type $(4,6)$ or of type $(6,6)$ (by \cref{res:edge_types}).
Each 1-vertex of $\sigma$ (which must be of type $(4,6,6)$) is then incident to exactly one of these $(6,6)$-edges of $\sigma$.
Thus, there are exactly \emph{three} $(6,6)$-edges incident to $\sigma$ (see \cref{fig:hexagon_fail}).
On the other hand, each 2-vertex of $\sigma$ is incident to an even number of $(6,6)$-edges of $\sigma$ (since if a 2-vertex is incident to at least one $(6,6)$-edge, then we have previously shown that its type must be $(6,6,6)$, implying another incident $(6,6)$-edge). Therefore the number of $(6,6)$-edges incident to $\sigma$ must be \emph{even} (see \cref{fig:hexagon_fail}), in contradiction to the previously obtained number three of such edges.
%Such a configuration of edge types is not possible around a 6-gon: the condition of the 1-vertices implies that there are \emph{exactly three} $(4,6)$-edges around $\sigma$, but the conditions of the 2-vertices imply that the number of $(4,6)$-edges is \emph{even} (see also \cref{fig:hexagon_fail}).

\begin{figure}
\centering
\includegraphics[width=0.55\textwidth]{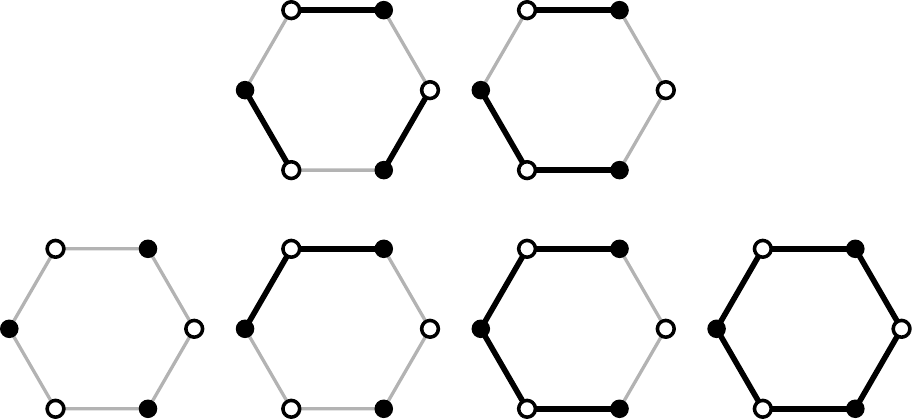}
\caption{Possible distributions of $(4,6)$-edges (gray) and $(6,6)$-edges (thick) around a 6-gon as discussed in \cref{sec:466}. The top row shows configurations compatible with the conditions set by 1-vertices (black), and the bottom row shows the configurations compatible with the conditions set by the 2-vertices (white).}
\label{fig:hexagon_fail}
\end{figure}
%\fi

Consequently, $P$~cannot contain a 1-vertex of type $(4,6,6)$.

\begin{observation}
It is a consequence of \cref{sec:466,sec:468,sec:4610} that $P$ cannot have a 1-vertex of a type $(4,6,2k')$ for a $2k'\in\{6,8,10\}$.
By \cref{res:incompatible_1_types} this means that \emph{all} 1-vertices of $P$ are of the same type $\tau_1=(4,4,2k)$ for some fixed $2k\ge 4$.
\end{observation}

It is worth to distinguish the case $(4,4,4)$ from the cases $(4,4,2k)$ with $2k\ge 6$.

\subsection{The case $\tau_1=(4,4,4)$}
\label{sec:444}

In this case, all 2-faces are 4-gons, and all 4-gons are congruent by \cref{res:unique_shape}.
A 4-gon with all edges of the same length is known as a \emph{rhombus}, and the polyhedra with congruent rhombic faces are known as \emph{rhombic isohedra} (from german \emph{Rhombenisoeder}).
These have a known classification:

\begin{theorem}[S. Bilinksi, 1960 \cite{bilinski1960rhombic}]\label{res:rhombic_isohedra}
If $P$ is a polyhedron with congruent rhombic faces, then $P$ is one of the following:
\begin{myenumerate}
	\item a member of the infinite family of rhombic hexahedra, \ie\ $P$ can be obtained from a cube by stretching or squeezing it along a long diagonal,
	\item the rhombic dodecahedron,
	\item the Bilinski dodecahedron,
	\item the rhombic icosahedron, or
	\item the rhombic triacontahedron.
\end{myenumerate}
\end{theorem}

The figure below depicts these polyhedra in the given order (from left to right; including only one instance from the family $(i)$):
\begin{center}
\includegraphics[width=0.9\textwidth]{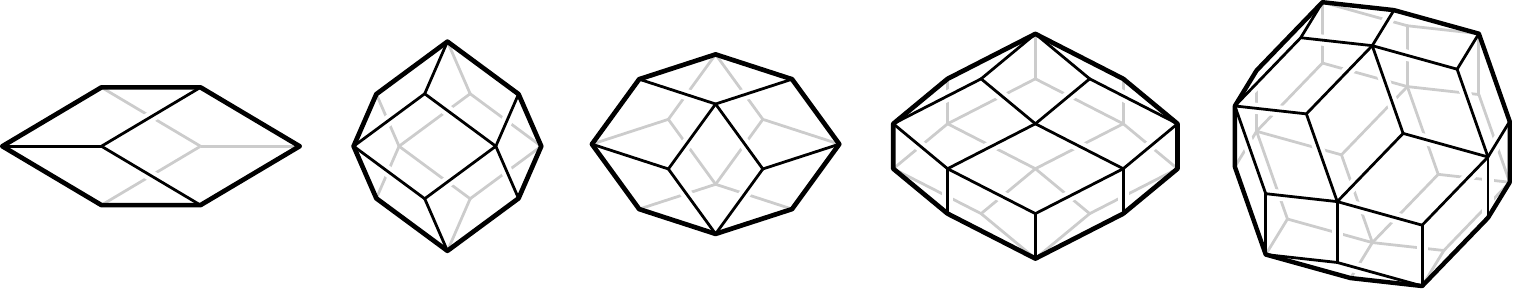}
\end{center}
The rhombic dodecahedron and triacontahedron are known edge- but not vertex-transitive polytopes.
We show that the others are not even strictly bipartite.

\begin{corollary}\label{res:strictly_bipartite_444}
If $P$ is strictly bipartite with all 1-vertices of type $(4,4,4)$, then~$P$ is one of the following:
\begin{myenumerate}
	\item the rhombic dodecahedron,
	\item the rhombic triacontahedron.
\end{myenumerate}
\begin{proof}
The listed ones are edge-transitive but not vertex-transitive.
Also they are not inscribed.
By \cref{res:trans_is_bipartite} they are therefore \emph{strictly} bipartite.

We then have to exclude the other polyhedra listed in \cref{res:rhombic_isohedra}.
The rhombic hexahedra include the cube, which is inscribed, hence not strictly bipartite.
In all the other cases, there exist vertices where acute and obtuse angles meet (see the figure).
So this vertex cannot be assigned to either $V_1$ or $V_2$ (\cf\ \cref{res:2k_4_case}), and the polyhedron cannot be bipartite.
\end{proof}
\end{corollary}

These are the only strictly bipartite polyhedra we will find, and both are edge-transitive without being vertex-transitive.

\subsection{The case $\tau_1=(4,4,2k),2k\ge 6$}
\label{sec:442k}

If $P$ contains a 1-vertex of type $(4,4,2k)$ with $2k\ge 6$, then it also has an adjacent pair of the form
$$(\tau_1,\tau_2)=((4,4,2k),(4,2k,*,...,*)).$$
We proceed as demonstrated in \cref{ex:infeasible}.
\cref{res:K_E_ineq} $(i)$ yields $\eps_2+\eps_2+\eps_k < 1/k$.
Since $(4,4,2k)$ is the only type of 1-vertex of $P$, there are only 4-faces and $2k$-faces and~the placeholders can only take on the values 4 and $2k$ (note that we do \emph{not} use \cref{res:exclusive_types} for this).
The following table lists some inequalities derived for infeasible pairs:
\begin{center}
\begin{tabular}{l|rcll}
$\tau_2$ & $s-2-K(\tau_2)$ &$\overset?<$&  $E(\tau_2)$ & \\[0.4ex]
\hline
%$\overset{\phantom.}(4,4,4,4,4,4)$ & $1$ &$<$& $(\eps_2+\eps_2)+(\eps_2+\eps_2)+(\eps_2+\eps_2)$ & $< 3/k$ \\
$\overset{\phantom.}(4,2k,4,4,4)$ & $1-1/k$ & $<$ & $(\eps_2+\eps_2+\eps_k)+(\eps_2+\eps_2)$ & $< 2/k$ \\
$(4,2k,4,4,2k)$ & $3/2-2/k$ &$<$& $(\eps_2+\eps_2+\eps_k)+(\eps_2+\eps_k)$ & $< 2/k$
\end{tabular}
\end{center}
One checks that these inequalities are not satisfied for $2k\ge 6$.
\Cref{res:subtype} then states that the placeholders can contain at most two 4-s, and if exactly two, then nothing else. 
Moreover, $\tau_2$ must contain at least as many 4-s as it con\-tains $2k$-s, as otherwise we would find two adjacent $2k$-faces while $P$ cannot contain a $(2k,2k)$-edge by \cref{res:edge_types}.
We are therefore left with the following options for $\tau_2$:
$$(4,4,2k),\; (4,4,4,2k) \quad\text{and}\quad (4,2k,4,2k).$$

The case $\tau_2=(4,4,2k)$ is impossible by \cref{res:1_2_vertex_same_type}.
We show that $\tau_2=(4,4,4,2k)$~is also not possible:
consider the local neighborhood of a $(4,4,4,2k)$-vertex (the highlighted~ver\-tex) in the following figure:
\begin{center}
\includegraphics[width=0.25\textwidth]{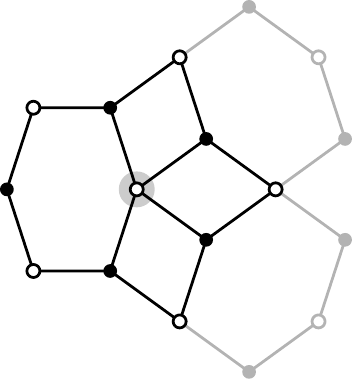}
\end{center}
Since the 1-vertices (black dots) are of type $(4,4,6)$, this configuration forces on us the~existence of the two gray 6-faces.
These two faces intersect in a 2-vertex, which is then incident to two $2k$-faces and must be of type $(4,2k,4,2k)$.
But we can show that the types $(4,4,4,2k)$ and $(4,2k,4,2k)$ are incompatible by \cref{res:spherical_interior_angles_are_good} $(i)$:
$$\beta_2^2+\beta_2^2+\beta_2^2+\beta_2^k \overset{(i)}= \beta_2^2+\beta_2^k+\beta_2^2+\beta_2^k \;\implies\;  \beta_2^2=\beta_2^k \;\overset{\eqref{eq:angle_relations}}\implies\; 4=2k\ge 6.$$
Thus, $(4,4,4,2k)$ cannot occur.

We conclude that every 2-vertex incident to a $2k$-face must be of type $(4,2k,4,2k)$.
%Now consider the third row in the above table (below the separator).
%
Consider then the following table:
\begin{center}
\begin{tabular}{l|rcll}
$\tau_2$ & $s-2-K(\tau_2)$ &$\overset?<$&  $E(\tau_2)$ & \\[0.4ex]
\hline
$\overset{\phantom.}(4,2k,4,2k)$ & $1-2/k$ &$<$& $(\eps_2+\eps_2+\eps_k)+\eps_2$ & $< 2/k$
\end{tabular}
\end{center}
The established inequality yields $2k\le 6$, and hence $2k=6$.
We found that then all 1-vertices must be of type $(4,4,6)$, and all 2-vertices incident to a 6-face must be of type $(4,6,4,6)$.

%
%\section{Strictly bipartite polyhedra of type $(4,4,2k)$}
%\label{sec:442k_polyhedra}
%
%In the process of classifying the strictly bipartite polyhedra, we found that all 1-vertices must necessarily be of type $(4,4,2k)$ for some fixed $2k\in\{4,6\}$.
%In this section, we finally show that only the rhombic dodecahedron and the rhombic triacontahedron remain.

\subsection{The case $\tau_1=(4,4,6)$}
\label{sec:446}

At this point we can now assume that all 1-vertices of $P$ are of type $(4,4,6)$ and that each 2-vertex of $P$ that is incident to a 6-face is of type $(4,6,4,6)$.
In particular, $P$ contains a 2-vertex $w\in V_2$ of this type.
Since there is no $(6,6)$-edge in $P$, the two 6-faces incident to $w$ cannot be adjacent.
In other words, the faces around $w$ must occur alternatingly of type 4 and type 6, which is the reason for writing the type $(4,6,4,6)$ with alternating entries.

On the other hand, $P$ contains $(4,4)$-edges, and none of these is incident to a $(4,6,4,6)$-vertex surrounded by alternating faces.
Thus, there must be further 2-vertices of a type other than $(4,6,4,6)$, necessarily \emph{not} incident to any 6-face.
These must then be of type
$$(4^r):=(\underbrace{4,...,4}_{r}),\quad\text{for some $r\ge 3$}.$$

\begin{proposition}
$r=5$.
\begin{proof}
If there is a $(4^r)$-vertex, \cref{res:spherical_interior_angles_are_good} $(i)$ yields $\beta_2^2=2\pi/r$.
Analogously,~from~the existence of a $(4,6,4,6)$-vertex follows
$$2\beta_2^2+2\beta_2^3\overset{(i)}=2\pi\quad\implies\quad \beta_2^3=\frac{2\pi-2\beta_2^2}2 = \Big(1-\frac2r\Big)\pi.$$
\begin{samepage}%
Recall $k\beta_1^k+k\beta_2^k>2\pi(k-1)$ from \cref{res:spherical_interior_angles_are_good} $(ii)$.
Together with the previously established values for $\beta_2^2$ and $\beta_2^3$, this yields %, for $2k=4$ \resp\ $2k=6$ this yields
\begin{equation}
\label{eq:4}
\begin{array}{rcl}
\beta_1^2 &\!\!\!>\!\!\!& \displaystyle \frac{2\pi(2-1)-2\beta_2^2}2= \Big(1-\frac2r\Big)\pi, \quad\text{and} \\[1ex]
\beta_1^3 &\!\!\!>\!\!\!& \displaystyle \frac{2\pi(3-1)-3\beta_2^3}3= \Big(\frac13+\frac2r\Big)\pi.
\end{array}
\end{equation}
%
%If $r\le 3$, then the right inequality yields $\beta_1^3>\pi$, which cannot be.
Since the 1-vertices are of type $(4,4,6)$, \cref{res:spherical_interior_angles_are_good} $(i)$ yields
\end{samepage}
$$2\pi\overset{(i)}=2\beta_1^2+\beta_1^3 \overset{\eqref{eq:4}}> 2\Big(1-\frac2r\Big)\pi + \Big(\frac13+\frac2r\Big)\pi = \Big(\frac73-\frac2r\Big)\pi.$$
And one checks that this rearranges to $r<6$. 

This leaves us with the options $r\in\{3,4,5\}$.
If $r=4$, then $\beta_2^3=\pi/2=\beta_2^2$, which~is~impossible by equation \eqref{eq:angle_relations}.
And if $r=3$, then \eqref{eq:4} yields $\beta_1^3 > \pi$, which is also impossible~for a convex face of a spherical polyhedron.
We are left with $r=5$.
\end{proof}
\end{proposition}

%To recollect, w
To summarize: $P$ is a strictly bipartite polyhedron in which~all 1-vertices are~of type $(4,4,6)$, and all 2-vertices are of types $(4,6,4,6)$ or $(4^5)$,~and both types actually occur in $P$.
%It remains to show that this is not possible.
This information turns out to be sufficient to uniquely determine the edge-graph of $P$, which is shown in \cref{fig:graph}.

%One can start by constructing the edge-graph of this hypothetical bipartite polyhedron, which is a planar graph.
%One finds that a graph with these properties exists and is unique:
%
\begin{figure}
\centering
\includegraphics[width=0.35\textwidth]{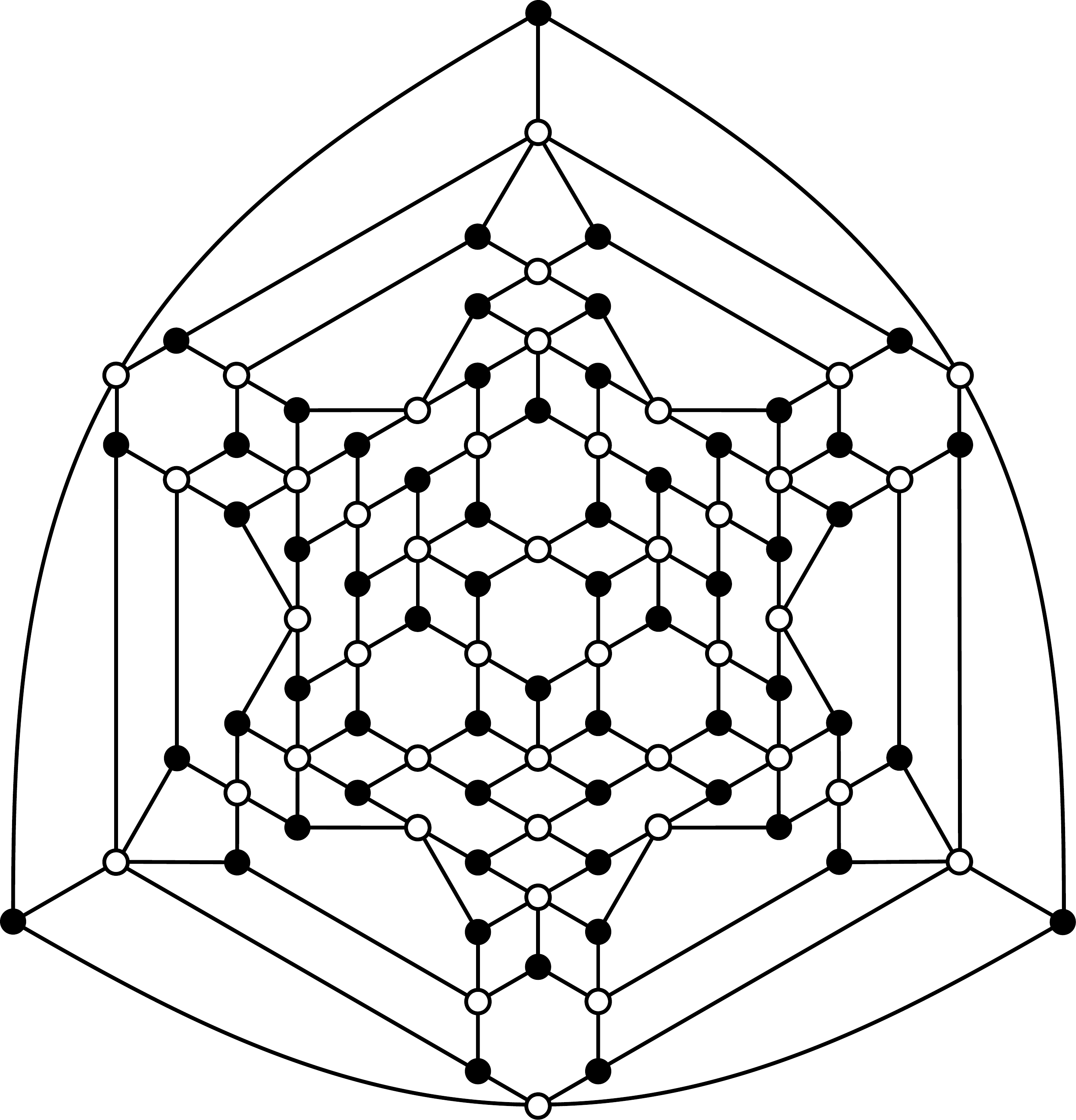}
\caption{The edge-graph of the final candidate polyhedron.}
\label{fig:graph}
\end{figure}
This graph can be constructed by starting with a hexagon in the center with vertices of alternating colors (indicating the partition classes).
One then successively adds further faces (according to the structural properties determined above), layer by layer.
This process involves no choice and thus the result is unique.

%We mentioned previously, that a bipartite polyhedron necessarily has an edge in-sphere (this is an alternative formulation of \cref{def:bipartite} $(iii)$).
As mentioned in \cref{rem:alternative_definiton}, a bipartite polyhedron has an edge in-sphere.
Thus, $P$ is a polyhedral realization of the graph in \cref{fig:graph} with an edge in-sphere.
It is known that any two such realizations are related by a projective transformation \cite{sachs1994coin}. %\cite[p.\ 117 - 118]{ziegler2012lectures}.
One representative $Q\subset\RR^3$ from this class (which we do not yet claim~to coincide with~$P$) can~be~constructed by
applying the following operation $\star$ to each vertex of the regular icosahedron: 
\begin{center}
\includegraphics[width=0.5\textwidth]{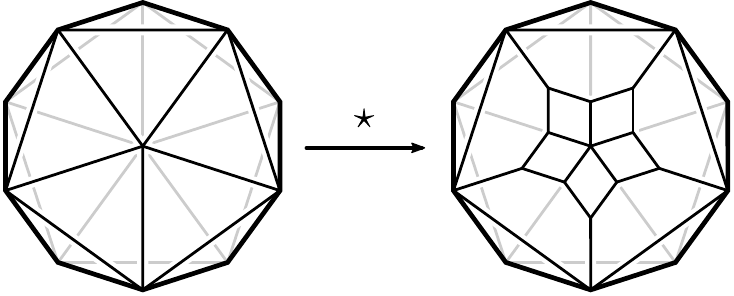}
\end{center}
The operation is performed in such a way, so that
\begin{itemize}
    \item the five new \enquote{outer} vertices of the new 4-gons are positioned in the centers of edges of the icosahedron.
    \item  the edges of each new 4-gon are tangent to a common sphere centered at the center of the icosahedron
\end{itemize} 
The resulting polyhedron $Q$ looks as follows:\\[-1.5ex]
\begin{center}
\includegraphics[width=0.7\textwidth]{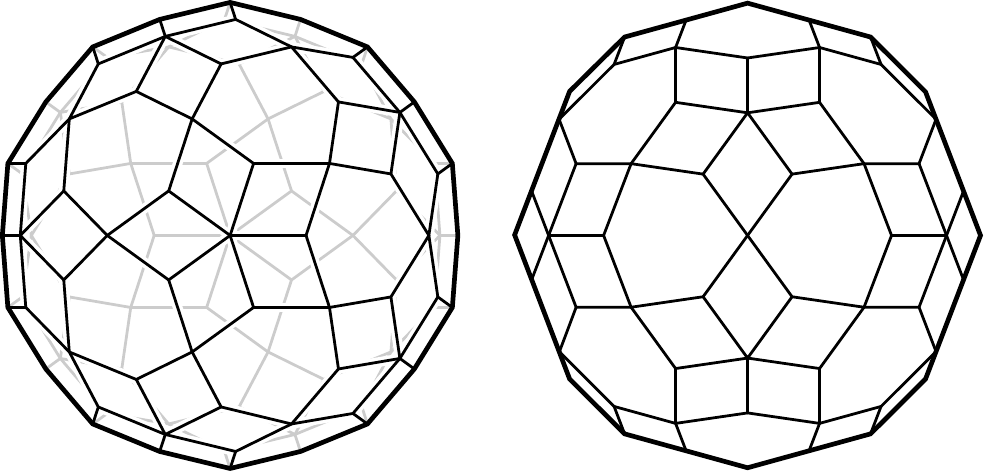}
\end{center}
One can verify that $Q$ has indeed the desired edge-graph.
%It furthermore belongs to the class of so-called \emph{symmetrohedra} and has Conway notation $\mathrm{dL_0 D}$ \cite{conway2016symmetries}.

It is clear from the construction that $Q$ has an edge in-sphere, and any two of its 4-gonal or 6-gonal faces are congruent (as we would expect from a bipartite polyhedron).
Like-wise, $P$ has an edge in-sphere and the same edge-graph. 
Hence, $P$ must be a projective transformation of $Q$.
However, any~projective transformation that is not just a re-orientation or a uniform rescaling will inevitably destroy the property of congruent faces.
In conclusion, we can assume that $P$ is identical to $Q$ (up to scale and orientation).

It remains to check whether $Q$ is indeed a bipartite~polyhedron. For this, recall that any two of the following properties imply the third (\cf\ \cref{rem:alternative_definiton}):
\begin{myenumerate}
    \item $Q$ has an edge in-sphere.
    \item $Q$ has all edges of the same length.
    \item for each vertex $v\in \F_0(Q)$, the distance $\|v\|$ only depends on the partition class of the vertex.
\end{myenumerate}
Now, $Q$ satisfies \itm1 by construction, and it would need to satisfy both \itm2 and \itm3 in order to be bipartite. 
The figure certainly suggests that all edges of $Q$ are of the same length.
However, as we shall show now, $Q$ cannot satisfy both \itm2 and \itm3 at the same time, and thus, can satisfy neither.
In particular, the edges must have a tiny difference in length that cannot be spotted visually, making $Q$ into a remarkable near-miss (we will quantify this below).

For what follows, let us assume that \itm2 holds, that is, that all edges of $Q$ are of the same length, in particular, that all 4-gons are rhombuses.
Our goal is to show that $\|v\|$ depends on the type of the vertex $v\in V_2$ (not only its partition class), establishing that \itm3 does not hold.

For this, start from the following well-known construction of the regular icosahedron from the cube of edge-length 2 centered at the origin.
\begin{center}
\includegraphics[width=0.95\textwidth]{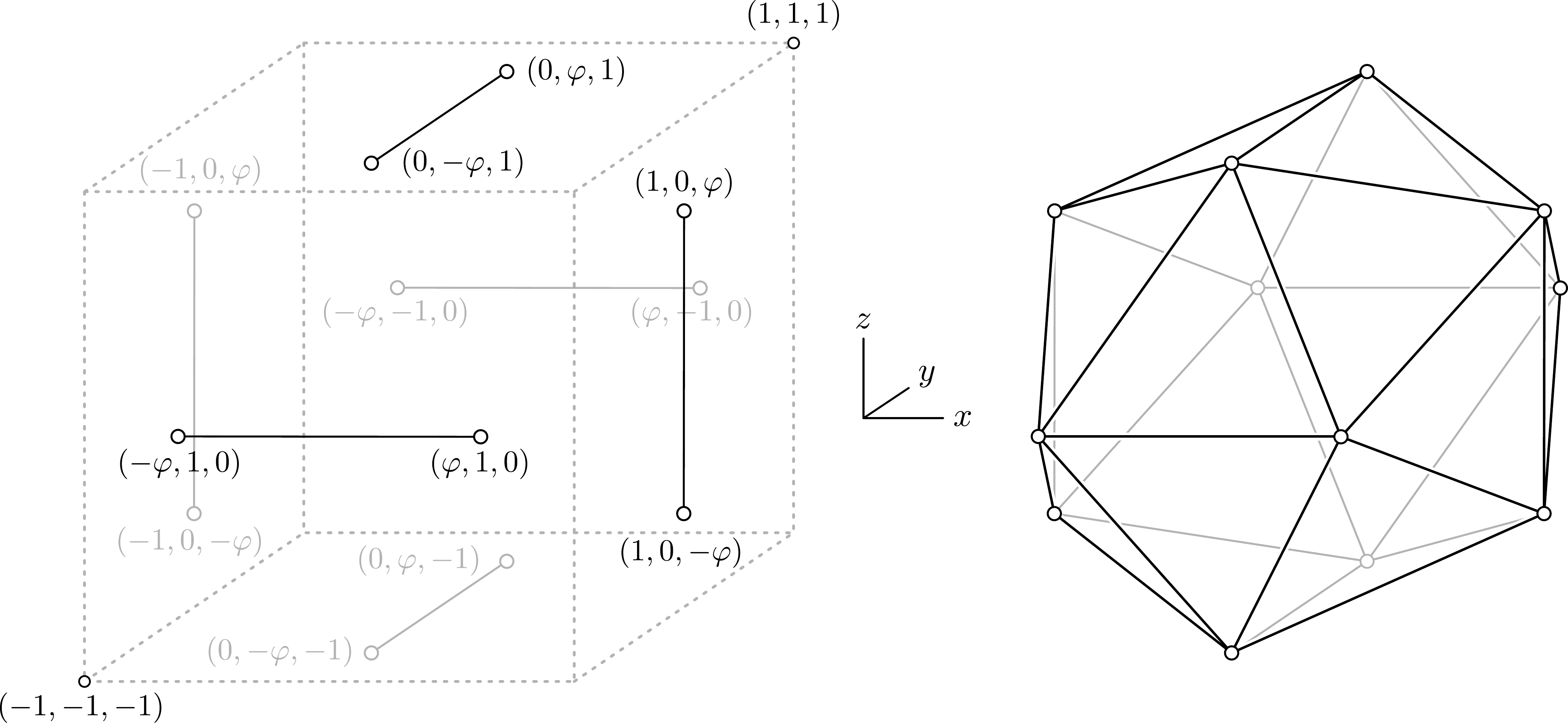}
\end{center}
The construction is as follows:
insert a line segment in the center of each face of the cube as shown in the left image. Each line segment is of length $2\varphi$, where $\varphi\approx 0.61803$ is the positive solution of $\varphi^2=1-\varphi$ (one of the numbers commonly knows as the \emph{golden ratio}).
The convex hull of these line segments gives the icosahedron with edge length $2\varphi$.

It is now sufficient to consider a single vertex of the icosahedron together with its incident faces.
The image below shows this vertex after we applied $\star$.
\begin{center}
\includegraphics[width=0.9\textwidth]{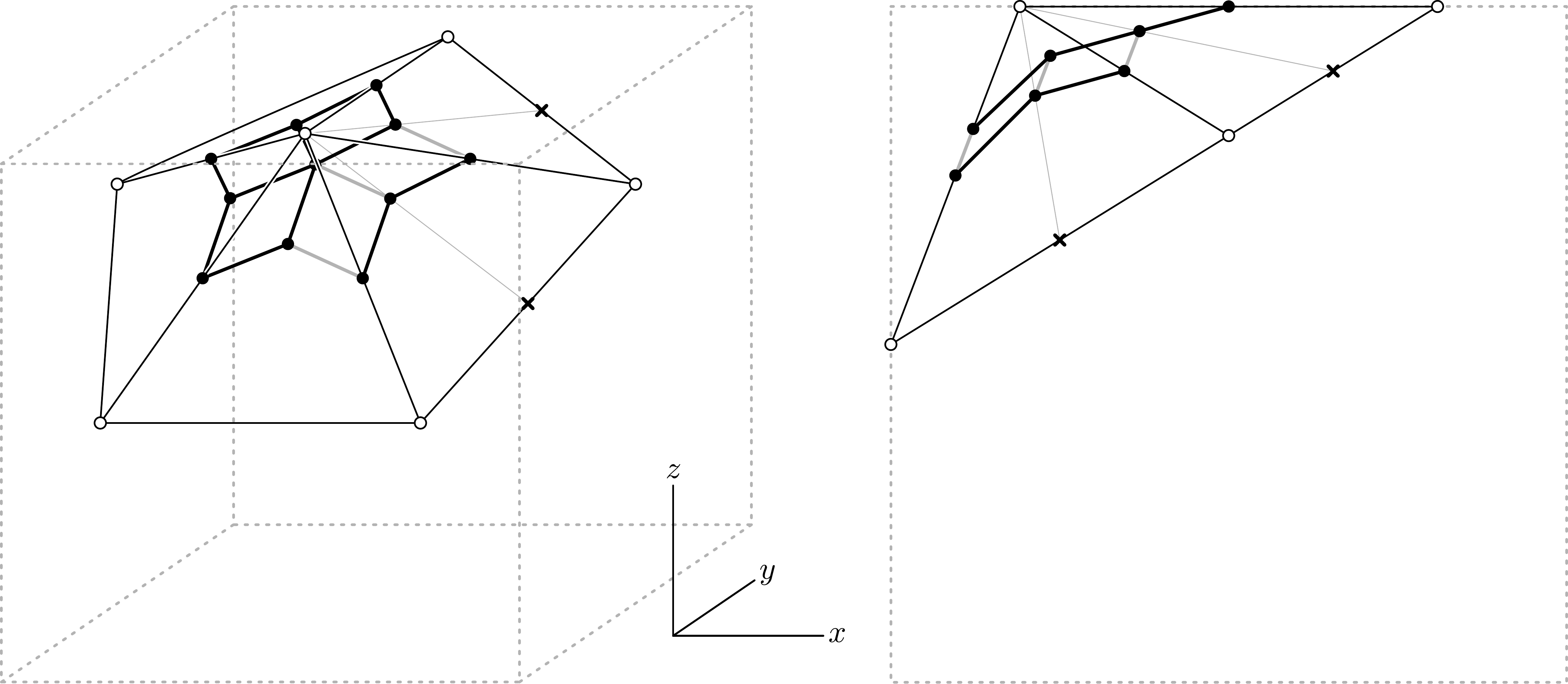}
\end{center}
The image on the right is the orthogonal projection of the configuration on the left onto the $yz$-plane.
%, or equivalently, onto the right face of the encapsulating cube.
%The picture also contains the additional faces that appear when we include the five halfspaces that define the rhombuses on this vertex.
%
This projection makes it especially easy to give 2D-coordinates for several important points:
\begin{center}
\includegraphics[width=0.85\textwidth]{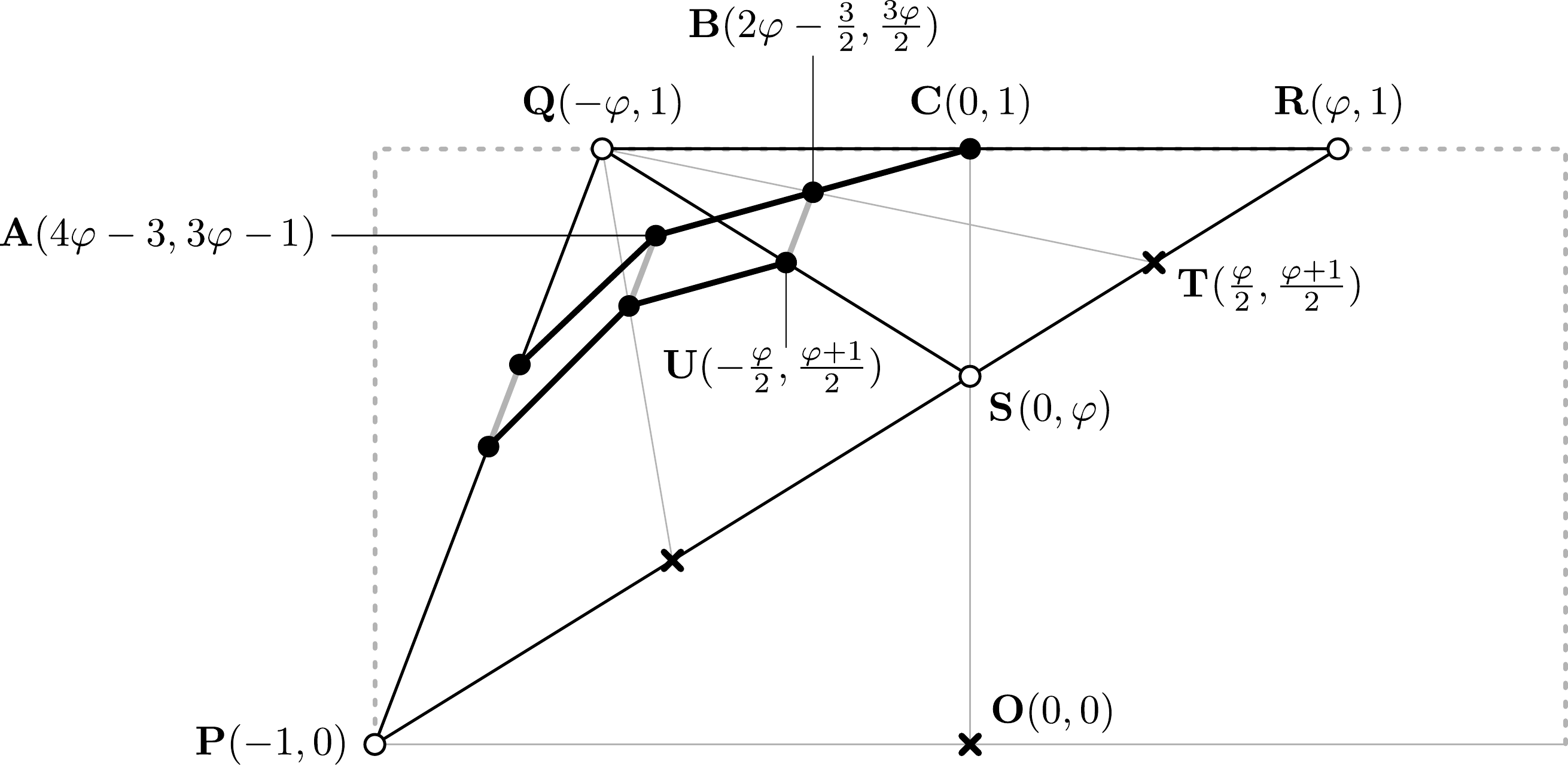}
\end{center}
The points $\mathbf A$ and $\mathbf C$ are 2-vertices of $Q$ of type $(4^5)$ and $(4,6,4,6)$ respectively.
%The point $\mathbf A$ is a 2-vertex of $P$ of type $(4^5)$, and the point $\mathbf C$ is a 2-vertex of type $(4,6,4,6)$. 
Both points and the origin $\mathbf O$ are contained in the $yz$-plane onto which we projected.
Consequently, distances between these points are preserved during the projection, and assuming that $Q$ is bipartite, we would expect to find $|\overline{\mathbf {OA}}|=|\overline{\mathbf{OC}}|=r_2$.
We shall see that this is not the case, by explicitly computing the coordinates of $\mathbf A$ and $\mathbf C$ in the new coordinate system $(y,z)$.

By construction, $\mathbf C=(0,1)$ and $|\overline{\mathbf{OC}}|=1$.
Other points with easily determined coordinates are $\mathbf P$, $\mathbf Q$, $\mathbf R$, $\mathbf S$, $\mathbf T$ (the midpoint of $\mathbf R$ and $\mathbf S$) and $\mathbf U$ (the midpoint of $\mathbf Q$ and $\mathbf S$).

By construction, the point $\mathbf B$ lies on the line segment $\overline{\mathbf{QT}}$.
The parallel projection of~a~rhombus is a (potentially degenerated) parallelogram, and thus, opposite edges in the projection are still parallel.
Hence, the gray edges in the figure are parallel.
For that reason, the segment $\overline{\mathbf{UB}}$ is parallel to $\overline{\mathbf{PQ}}$.
This information suffices to determine the coordinates of $\mathbf B$, which is now the intersection of $\overline{\mathbf{QT}}$ with the parallel of $\overline{\mathbf{PQ}}$ through $\mathbf U$.
The coordinates are given in the figure.

The rhombus containing the vertices $\mathbf A$, $\mathbf B$ and $\mathbf C$ degenerated to a line.
Its~fourth vertex is also located at $\mathbf B$. 
Therefore, the segments $\overline{\mathbf{CB}}$ and $\overline{\mathbf{BA}}$ are translates of each other.
Since the point $\mathbf B$ and the segment $\overline{\mathbf{CB}}$ are known, this allows the computation of the coordinates of $\mathbf A$ as given in the figure.

We can finally compute $|\overline{\mathbf{OA}}|$.
For this, recall $(*)\,\varphi^{2n}=F_{2n-2}-\varphi F_{2n-1}$, where $F_n$ denotes the $n$-th \emph{Fibonacci number} with initial conditions $F_0=F_1=1$. 
Then
\begin{align*}
|\overline{\mathbf O\mathbf A}|^2 
&= (4\varphi-3)^2+(3\varphi-1)^2 \\
%&= 16\varphi^2-24\varphi + 9 + 9\varphi^2 - 6\varphi + 1 \\
&= 25\varphi^2 - 30\varphi + 10 \\[-1ex]
&\overset{\mathclap{(*)}}= 25(1-\varphi)-30\varphi + 10 \\
&= 35 - 55\varphi 
\\[-1ex] &= 1 + (34 - 55\varphi)\overset{\mathclap{(*)}}=1+\varphi^{10} > 1,
\end{align*}
and thus, $Q$ cannot be bipartite.
Remarkably, we find that 
$$|\overline{\mathbf O\mathbf A}|=\sqrt{1+\varphi^{10}}\approx 1.00405707$$
is only about $0.4\%$ larger than $|\overline{\mathbf O\mathbf C}|=1$, and so while $Q$ is not bipartite, it~is~a~remarkable near-miss.

%\begin{theorem}
%If $P\subset\RR^3$ is a strictly bipartite polyhedron, then $P$ is the rhombic dodecahedron or the rhombic triacontahedron.
%\end{theorem}

Since $P$ was assumed to be bipartite, but was also shown to be identical to $Q$, we reached a contradiction, which finally proves \cref{res:strictly_bipartite_polyhedra}, and the goal of the paper is achieved.

%This is actually more than what we need: it is actually much easier to show that such a polyhedron, assuming that it exists, cannot be the face of a 4-dimensional strictly bipartite polyhedron.
%The argument is the same as in the proof of \cref{res:...} by observing that the dihedral angles are too larger.
%
%\begin{proposition}
%All interior angles of $P_5$ are $\ge 120^\circ$.
%%
%\begin{proof}
%We have $\beta_2^2=2\pi/5$, and $2\beta_2^2+2\beta_2^3=2\pi\implies \beta_2^3=3\pi/5$.
%Then follows $\alpha_2^2\le 2\pi/5\implies \alpha_1^2 \ge 3\pi/5$ and $\alpha_2^3\le 3\pi/5\implies\alpha_1^3\ge 11\pi/15$.
%
%So if $v\in V_1$ is a 1-vertex with neighbors $w_1,w_2,w_3\in V_2$ and $$u_i:=\frac{w_i-v}{\|w_i-v\|}$$ the normalized edge directions, then $\<u_i,u_j\>\le \cos(3\pi/5)$ if the enclosed face is a 4-gon, and $\le \cos(11\pi/15)$ if the enclosed face is a 6-gon.
%In other words, the matrix
%
%\end{proof}
%\end{proposition}
%
%However, we shall go the long way.
%Not only has above result a numerical flavored imprecision, it is also much more satisfying to have all bipartite polyhedra classified.
%Furthermore, we will see that while a polyhedron of type $(4^5)$ cannot be bipartite, it nevertheless comes \emph{very} close, and can be considered a remarkable near miss.
%\input{sec/zonotopes}
%\input{sec/content}
%\input{sec/future}

\section{Conclusions and open questions}
\label{sec:conclusions}

In this paper we have shown that any edge-transitive (convex) polytope in four or more dimensions is necessarily vertex-transitive.
We have done this by classifying all polytopes which simultaneously have all edges of the same length, an edge in-sphere and a bipartite edge graph (which we named \emph{bipartite} polytopes).

The obstructions we~derived for being edge-transitive without being vertex-transitive have been mainly geometrical and less a matter of symmetry (a detailed investigation of the Euclidean symmetry groups was not necessary, but it might be interesting to view the problem from this perspective).
We suspect that dropping convexity or considering combinatorial symmetries instead of geometrical ones will quickly lead to further examples of just edge-transitive structures.
For example, it is easy to find embeddings of graphs into $\RR^d$ with these properties.

Slightly stronger than being simultaneously vertex- and edge-transitive, is being transitive on~\emph{arcs}, that is, on incident vertex-edge pairs.
This additional degree of symmetry allows an edge to be not only mapped onto any other edge, but also onto itself with inverted orientation.
While there are graphs that are vertex- and edge-transitive without being arc-transitive (the so-called \emph{half-transitive} graphs, see \cite{holt1981graph}), we believe it is unlikely that this distinction is necessary for convex polytopes.

\begin{question}
Is there a polytope $P\subset\RR^d$ that is edge-transitive and vertex-tran\-sitive,~but~not arc-transitive?
\end{question}

In a different direction, the questions of this paper naturally generalize to faces of higher dimensions. In general, the interactions between transitivities of faces of different dimensions have been little investigated.
%: Horst Martini \cite{martini1994hierarchical}  writes
%
%\begin{quote}
%More generally, one can consider $k$-transitivity for each $k\in\{0,1,...,d-1\}$.
%Among the various questions concerning this notion, the relation between the transitivites of different dimensions deserves to be investigated.
%\end{quote}
%
%The results of this paper can be seen as a first step in this direction.
For example, already the~following question seems to be open:
\begin{question}
For fixed $k\in\{2,...,d-3\}$, are there convex $d$-polytopes for arbitrarily large $d\in\NN$ that are transitive on $k$-dimensional faces without being transitive on either vertices or facets?
\end{question}

Of course, any such question could be attacked by attempting to classify the $k$-face-transitive (convex) polytopes for some $k\in\{1,...,d-2\}$.
It seems to be~un\-clear for which $k$ this problem is tractable (for comparison, $k=0$ is intractable, see \cite{babai1977symmetry}), and it appears that there are no techniques applicable to all (or many) $k$ at the same time.

\vspace{1em}
\textbf{Acknowledgements.} 
%The second author gratefully acknowledges the support by the funding of the European Union and the Free State of Saxony (ESF).
%Furthermore, the author thanks Frank Göring (TU Chemnitz) for providing the detailed computation to determine the geometric properties of the remarkable near miss polyhedron in \cref{sec:446}.
The authors thank the anonymous referee for his careful reading and his many remarks that led to the improvement of the article in several places.

%%%%%%%%%%%%%%%%%%%%%%%%%%%%%%%%%%%%%%%%%%%%%%%%%%%%%%%%%%%%%%%%%%%%%%%%%%%%%%%%%%

\bibliographystyle{unsrt}
\bibliography{literature}

\newpage

\appendix

\section{}
\label{sec:appendix}

\subsection{Geometry}

\begin{proposition}\label{res:sum_to_zero}
Given a set $x_0,...,x_d\in\RR^d\nozero$ of $d+1$ vectors with pair-wise negative inner product,
then there are \ul{positive} coefficients $\alpha_0,...,\alpha_d>0$ with
$$\alpha_0 x_0+\cdots +\alpha_d x_d=0.$$
\iftrue
\begin{proof}
We proceed by induction. 
The induction base $d=1$ which is trivially true.
 
Now suppose $d\ge 2$, and, W.l.o.g.\ assume $\|x_0\|=1$.
Let $\pi_0$ be the orthogonal projection onto $x_0^\bot$, that is, $\pi_0(u):= u - x_0\<x_0,u\>$. 
In particular, for $i\not=j$ and $i,j>0$
$$\<\pi_0(x_i),\pi_0(x_j)\> = \underbrace{\<x_i,x_j\>}_{<0} - \underbrace{\<x_0,x_i\>}_{<0}\underbrace{\<x_0,x_j\>}_{<0} < 0.$$
Then $\{\pi(x_1),\...,\pi_0(x_d)\}$ is a set of $d$ vectors in $x_0^\bot\cong\RR^{d-1}$ with pair-wise negative inner product.
By induction assumption there are positive coefficients $\alpha_1,...,\alpha_d>0$ so that $\alpha_1\pi_0(x_1)+\cdots+\alpha_d\pi_0(x_d)=0$.

Set $\alpha_0:=-\<x_0,\alpha_1x_1+\cdots+\alpha_d x_d\>>0$.
We claim that $x:=x_0\alpha_0+\cdots+\alpha_dx_d=0$.
Since $\RR^d=\Span\{x_0\}\oplus x_0^\bot$, it suffices to check that $\<x_0,x\>=0$ as well as $\pi_0(x)=0$. This follows:
\begin{align*}
\<x_0,x\> &= \alpha_0\underbrace{\<x_0,x_0\>}_{=1} + \underbrace{\<x_0,\alpha_1x_1+\cdots+\alpha_d x_d\>}_{=-\alpha_0}=0, \\
\pi_0(x) &= \alpha_0 \underbrace{\pi_0(x_0)}_{=0} + \underbrace{\alpha_1\pi_0(x_1)+\cdots+\alpha_d\pi_0(x_d)}_{=0}=0.
\end{align*}
\end{proof}

\else

\begin{proof}
Let $X=(x_0,...,x_d)\in\RR^{d\x(d+1)}$ be the matrix with the $x_i$ as columns.
There are more vectors than dimensions, and thus, the $x_i$ must be linearly dependent, \ie\ there is an $a\in\RR^{d+1}$ with $Xa=0$.
We need to show that $a$ can be chosen so, that all entries are positive.

Consider the matrix $Y:=X\T\! X\in\RR^{(d+1)\x(d+1)}$, which is positive semi-definite and satisfies $Ya=X\T\!(Xa)=0$.
Since eigenvalues of positive semi-definite matrices are non-negative, we found that the smallest eigenvalue of $Y$ is zero, and $a$ is an associated eigenvector.

We have $Y_{ij}=\<x_i,x_j\><0$ for $i\not=j$, and thus, for an appropriately chosen $\eps>0$, $Z:=I-\eps Y$ is a matrix with only positive entries.
Each eigenvalue of $Z$ is of the form $1-\eps \lambda$, where $\lambda$ is an eigenvalue of $Y$.
In particular, the maximal eigenvalue of $Z$ is $1$, and $a$ is an associated eigenvector.
By \emph{Perron-Frobenius}, the entries of $a$ are non-zero and have all the same sign.
In particular, they can be chosen positive.
\end{proof}
\fi
\end{proposition}

\begin{proposition}\label{res:simple_vertex}
Let $P\subset\RR^3$ be a polyhedron with $v\in\F_0(P)$ a vertex of degree three.
The interior angles of the faces incident to $v$ determine the dihedral angles at the edges incident to $v$ and vice versa.
\begin{proof}
For $w_1,w_2,w_3\in\F_0(P)$ the neighbors of $v$, let $u_i:=w_i-v$ denote the direc\-tion of the edge $e_i$ from $v$ to $w_i$.
Let $\sigma_{ij}$ be the face that contains $v,w_i$ and $w_j$.
Then $\angle(u_i,u_j)$ is the interior angle of $\sigma_{ij}$ at $v$.

The set $\{u_1,u_2,u_3\}$ is uniquely determined (up to some orthogonal transformation) by the angles $\angle(u_i,u_j)$.
Furthermore, since $P$ is convex, $\{u_1,u_2,u_3\}$ forms a basis of $\RR^3$, and this uniquely determines the \emph{dual basis} $\{n_{12}, n_{23}, n_{31}\}$ for which $\<n_{ij},u_i\>=\<n_{ij},u_j\>=0$.
In other words, $n_{ij}$ is a normal vector to $\sigma_{ij}$.
The dihedral angle at the edge $e_j$ is then $\pi-\angle(n_{ij},n_{jk})$, hence uniquely determined.
The other direction is analogous, via constructing $\{u_1,u_2,u_3\}$ as the dual basis to the set of normal vectors. %\TODO \msays{why are the normals outwards pointing?}
\end{proof}
\end{proposition}

%\begin{proposition}
%Given a simple $S:=\conv\{a,b,c,d\}$, so that $\|a-b\|=\|a-c\|$ and $\|d-b\|=\|d-c\|$.
%Then the interior angle of the face $\conv\{a,b,c\}$ at $a$, as well as the interior angle of the face $\conv\{b,c,d\}$ at $d$, is smaller than the dihedral angle of $S$ at the edge $\conv\{a,b\}$. 
%%
%\begin{proof}
%Let $e$ the direction of the edge
%\end{proof}
%\end{proposition}

\subsection{Computations}
\label{sec:computations}

The edge lengths in a spherical polyhedron are measured as angles between its end vertices. Consider adjacent vertices $v_1^S,v_2^S\in\F_0(P^S)$, then the incident edge has (arc-)length $\ell^S:=\angle(v_1^S,v_2^S)=\angle(v_1,v_2)$.

It follows from \cref{res:angles_between_vertices} that these angles are completely determined by the parameters, hence the same for all edges of $P^S$.

\begin{proposition}\label{res:computation}
For a face $\sigma\in\F_2(P)$ and a vertex $v\in\F_0(\sigma)$, there is a direct relationship between the value of $\alpha(\sigma,v)$ and the value of $\beta(\sigma,v)$.
\begin{proof}
Let $w_1,w_2\in V_2$ be the neighbors of $v$ in the $2k$-face $\sigma$, and set $u_i:=w_i-v$. Then $\angle(u_1,u_2)=\alpha(\sigma,v)$.
W.l.o.g.\ assume that $v$ is a 1-vertex (the argument is equivalent for a 2-vertex).
%We show that the interior angle $\alpha_1^k$ of the triangle $w_1 v w_2$ at $v$ uniquely determines the (spherical) interior angle $\beta_1^k$ of the spherical triangle $w_1^S v^S w_2^S$ at $v^S$, and vice versa.

%Set $u_i:=w_i-v\in\RR^3$, so that $\alpha_i^k=\angle(u_1,u_2)$ and $\|u_i\|=\ell$.
For convenience, we introduce the notation $\chi(\theta):= 1-\cos(\theta)$.
%Then $\chi$ is an order preserving bijection, which is convenient when using it to compare angles. 
We find that
%
%\begin{align*}
%\|w_1-w_2\|^2&=\|u_1\|^2+\|u_2\|^2-2\|u_1\|\|u_2\|\cos(\alpha) = \ell^2\cdot \chi(\alpha)\\
%%
%\|w_1-w_2\|^2&= \|w_1\|^2+\|w_2\|^2-\|w_1\|\|w_2\|\cos\angle(w_1,w_2) = r_2^2\cdot\chi(\angle(w_1,w_2)).
%\end{align*}
\begin{align*}
(*)\quad 2\ell^2\cdot\chi(\alpha(\sigma,v)) 
&= \ell^2+\ell^2-2\ell^2\cos(\angle(u_1,u_2)) \\
&= \|u_1\|^2+\|u_2\|^2 - 2\<u_1,u_2\> \\
&= \|u_1-u_2\|^2 
 %= \|w_1-v-w_2+v\|^2 
 = \|w_1-w_2\|^2 \\
&= \|w_1\|^2+\|w_2\|^2-2\<w_1,w_2\>\\
&= r_2^2+r_2^2-r_2^2\cos\angle(w_1,w_2) 
 = 2r_2^2\cdot\chi(\angle(w_1,w_2)).
\end{align*}

%Consider the spherical triangle $w_1^S v^S w_2^S$, whose angle at $v^S$ is exactly the interior angle of the face $\sigma^S$ at $v^S$.
%Lets call this angle $\beta_1^k$ for now.
The side lengths of the spherical triangle $w_1^Sv^Sw_2^S$ are $\angle(w_1,w_2), \ell^S$ and $\ell^S$.
By the spherical law of cosine\footnote{$\cos(c)=\cos(a)\cos(b)+\sin(a)\sin(b)\cos(\gamma)$, where $a,b$ and $c$ are the side lengths (arc-lengths) of a spherical triangle, and $\gamma$ is the interior angle opposite to the side of length $c$.} we obtain
\begin{align*}
\cos\angle(w_1,w_2) 
&= \cos(\ell^S)\cos(\ell^S) + \sin(\ell^S)\sin(\ell^S)\cos(\beta(\sigma,v)) \\
&= \cos^2(\ell^S) + \sin^2(\ell^S) (\cos(\beta(\sigma,v))-1+1) \\
&= [\cos^2(\ell^S) + \sin^2(\ell^S)] + \sin^2(\ell^S) (\cos(\beta(\sigma,v))-1) \\
&= 1-\sin^2(\ell^s)\cdot\chi(\beta(\sigma,v))\\
\implies\quad \sin^2(\ell^S)\cdot \chi(\beta(\sigma,v)) &= \chi(\angle(w_1,w_2))  \overset{(*)}= \Big(\frac{\ell}{r_2}\Big)^2\cdot \chi(\alpha(\sigma,v)) .
\end{align*}
%
%This rearranges to $\sin^2(\ell^S)\cdot\chi(\beta_1^k)=\chi(\angle(w_1,w_2)) \overset{(*)}= \ell^2/r_2^2\cdot  \chi(\alpha_1^k)$.
%W.l.o.g.\ this argument generalizes to \TODO
%%
%$$
%\sin^2(\ell^S)\cdot\chi(\beta_i^k)
%=
%\ell^2/r_ {(1-i)}^2\cdot  \chi(\alpha_i^k)$$
%
%We found a direct relation between $\beta_1^k$ and $\alpha_1^k$, each one determining the other one uniquely.
%The argument is equivalent for $\beta_2^k$.
\end{proof}
\end{proposition}

\end{document}